\documentclass[10pt]{amsart}
\usepackage{threeparttable}
\usepackage{mathtools}
\usepackage[ansinew]{inputenc}\usepackage[T1]{fontenc}
\usepackage[all,cmtip]{xy}
\usepackage{tikz} 
\usetikzlibrary{shapes.geometric, arrows}  
\usepackage{geometry}
\usepackage{amscd,amssymb,verbatim,xcolor}
\usepackage{longtable, multirow}
\usepackage[all,cmtip]{xy}
\date{\today}

\newcommand{\fra}{\mathfrak{a}}

\newcommand{\frg}{\mathfrak{g}}
\newcommand{\frh}{\mathfrak{h}}

\newcommand{\frk}{\mathfrak{k}}
\newcommand{\frl}{\mathfrak{l}}

\newcommand{\frs}{\mathfrak{s}}
\newcommand{\frt}{\mathfrak{t}}

\newtheorem{theorem}{Theorem}[section]

\newtheorem{corollary}[theorem]{Corollary}

\newtheorem{example}[theorem]{Example}
\newtheorem{lemma}[theorem]{Lemma}
\newtheorem{proposition}[theorem]{Proposition}

\theoremstyle{definition}
\newtheorem{definition}[theorem]{Definition}
\newtheorem{remark}[theorem]{Remark}

\begin{document}

\setlength{\baselineskip}{1.2\baselineskip}
\title[Lefschetz Principle for $\mathrm{GL}(n,\mathbb{C})$ and $\mathrm{GL}(m,\mathbb{Q}_p)$]
{On the Lefschetz Principle for $\mathrm{GL}(n,\mathbb{C})$ and $\mathrm{GL}(m,\mathbb{Q}_p)$}
\author[Kei Yuen Chan]{Kei Yuen Chan}
\author[Kayue Daniel Wong]{Kayue Daniel Wong}
\address{Department of Mathematics, The Unviersity of Hong Kong}
\email{kychan1@hku.hk}
\address{School of Science and Engineering, The Chinese University of Hong Kong, Shenzhen, Guangdong 518172, China}
\email{kayue.wong@gmail.com}

\begin{abstract}
We construct an exact functor from the category of Harish-Chandra modules of $\mathrm{GL}_n(\mathbb C)$ to the category of finite-dimensional modules of a graded Hecke algebra of type A. We show that the functor preserves parabolically induced modules, standard modules, irreducible modules, unitary modules and Dirac series. We also use the functor to connect a Bernstein-Zelevinsky type functor for graded Hecke algebra side to the tensor product for $\mathrm{GL}_n(\mathbb C)$ side. As an application, we explain how to obtain some new instances of explicit quotient branching laws from our results.

\end{abstract}

\maketitle

\section{Introduction} \label{s intro}

The Harish-Chandra Lefschetz principle predicts that the representation theory for real, $p$-adic and automorphic ones should be put in equal footing. This principle has been pursued in various situations. For example, the spherical unitary dual problem has been exemplified by Barbasch in \cite{B10}, in which the correspondence between relevant $K$-types (real side) and relevant $W$-types ($p$-adic side) are established to match up actions of intertwining operators.

Another instance is the work of Ciubotaru-Trapa \cite{CT11, CT12}, which builds functorial connections between the categories for some real and $p$-adic groups, based on some Schur-Weyl type constructions of Arakawa-Suzuki \cite{AS98} as well as the work of  Oda \cite{Od07} and Etingof-Freund-Ma \cite{EFM08}. This allows one to transfer information between these two categories, including the unitarity of representations.

The main goal of this paper is to establish a new Schur-Weyl type duality between the category of Harish-Chandra modules of $\mathrm{GL}_n(\mathbb C)$ and the category of modules of graded Hecke algebra of type A. While the work is inspired by \cite{AS98, Od07, EFM08, CT11, CT12}, there are two important aspects in our work: 
\begin{itemize}
    \item Firstly, the construction for $\mathrm{GL}_n(\mathbb C)$ carries other subtleties such as finding a suitable choice of the Casimir element in constructing Hecke algebra actions and a suitable choice of a standard representation in defining our functor.
    \item Secondly, other than the unitary dual problem -- which is much known now from the work of Barbasch \cite{Ba89} for complex case (also see \cite{St67} and \cite{Vo86}) and Tadi\'c \cite{Ta86} for $p$-adic case, we also explore other problems of recent interest such as Dirac series \cite{DW}, parabolic inductions \cite{LM16} and Bernstein-Zelevinsky derivatives \cite{Ch22+, Ch22+d} (related to branching laws).
 \end{itemize}
 
We need more notations to explain our constructions and main results. Let $\mathcal{HC}_n$ be the category of Harish-Chandra modules of $\mathrm{GL}_n(\mathbb C)$. For the $p$-adic side, the classical result of Borel-Casselman \cite{Bo76} allows to reduce the study of the Iwahori component of $p$-adic groups to the module category of affine Hecke algebras. Lusztig \cite{Lu89} further shows that one can further reduce to study the module category of their infinitesimal version -- graded Hecke algebras. In our context, we need the graded Hecke algebra of type $A$, see Definition \ref{def gaha type a}. Let $\mathcal H_m$ be the category of $\mathbb H_m$-modules.

Let $V$ be the {\it conjugate} standard representation of $\mathrm{GL}_n(\mathbb C)$ (see Lemma \ref{lem std repn}). Let $K$ be the maximal compact subgroup in $\mathrm{GL}_n(\mathbb C)$. In Section \ref{sec-as}, we construct a functor $\Gamma_{n,m}: \mathcal{HC}_n \rightarrow \mathcal{H}_m$ whose underlying space takes the form:
\[   \Gamma_{n,m}(X) =\mathrm{Hom}_{K}(\mathrm{triv}, X\otimes V^{\otimes m}) .
\]
Such functor has several nice behaviours:

\begin{theorem} \label{thm functor preserve properties}
The functor $\Gamma_{n,m}$ satisfies the following properties:
\begin{itemize}
   \item (Theorem \ref{cor preserve parabolic induction}) $\Gamma_{n,m}$ preserves parabolic inductions.
       \item (c.f. Theorem \ref{thm-std}) $\Gamma_{n,m}$ sends a standard representation to a standard $\mathbb H_m$-module or zero;
    \item (c.f. Theorem \ref{thm irreducibility of functor}) $\Gamma_{n,m}$ sends an irreducible module to an irreducible $\mathbb H_m$-module or zero;
    \item (Theorem \ref{thm preserve hermitian structure}) $\Gamma_{n,m}$ preserves unitarity.
\end{itemize}
\end{theorem}
We refer the reader to the corresponding statements for the precise meaning of the preservations in Theorem \ref{thm functor preserve properties}. We also remark that up to a thickening trick (see Section \ref{s preserve irr}), for any irreducible Harish-Chandra module $X$ of $\mathrm{GL}_n(\mathbb C)$, we can find a unique $m$ such that $\Gamma_{n,m}(X)$ is non-zero (and so irreducible). This nature is specific to $\mathrm{GL}_n(\mathbb C)$, compared to the $\mathrm{GL}_n(\mathbb R)$ case in \cite{CT12}.

In view of the recent development of parabolic induction on $p$-adic side (see e.g. \cite{LM16}), one can transfer those properties using our functor in an effortless way (see Section \ref{s application on parabolic}). This avoids at least some reworking for the parabolic induction in the Harish-Chandra category, not mentioning techniques in these two categories are different. Indeed, our functor also reveals what the analogue should be, which is sometimes not completely trivial.

A refinement on the unitary dual is the Dirac series, which carries deeper structure between $K$-types and $W$-types. Combining the classification of the Dirac series in \cite{BP}, \cite{DW} and \cite{BC14}, we show that up to thickening, $\Gamma_{n,m}$ also exhibits the Lefschetz principle for the Dirac series in Section \ref{s lefeschetz principle dirac}. This can be regarded as an answer to a question posted in \cite[Page 200]{BCT} for our setting. Extensions to some non-unitary representations with non-zero Dirac cohomology are also discussed. 

One classical application of the Schur-Weyl duality is to link up the branching problem for symmetric groups $S_i$ to the tensoring problem for $\mathfrak{gl}_n(\mathbb C)$. To pose such analogous linkage in our context, we need to introduce two more ingredients. Firstly, using a natural subalgebra $\mathbb C[S_i]$ in $\mathbb H_m$, we can define a functor: for an irreducible representation $\tau$ of $S_i$,
\[   \mathbf{BZ}_{\tau}: \mathcal H_n \rightarrow \mathcal H_{n-i}; \quad \mathbf{BZ}_{\tau}(\pi)=\mathrm{Hom}_{S_i}(\tau, \pi),
\]
see Section \ref{ss generalized bz} for precise descriptions. Secondly, we define the Schur functor $\mathbb S_{\tau}(V)=\mathrm{Hom}_{S_i}(\tau, V^{\otimes i})$ as a $\mathrm{GL}_n(\mathbb C)$-representation, which is irreducible and can be computed explicitly from the classical Schur-Weyl duality. (Here we use the convention that $S_i$ acts on $V^{\otimes i}$ by a sign permutation.) This gives the tensoring functor $T_{\tau}: \mathcal{HC}_n \rightarrow \mathcal{HC}_n$; $X \mapsto X\otimes \mathbb S_{\tau}(V)$. Our main result is to relate these two functors:

\begin{theorem} (=Theorem \ref{thm bz undr arakawa suzuki}) \label{thm bz gamma tensor}
Let $\tau \in \mathrm{Irr}(S_i)$. For $X \in \mathcal{HC}_n$, there is a natural isomorphism:
\[      \mathbf{BZ}_{\tau} \circ \Gamma_{n,m+i}(X) \cong \Gamma_{n,m}\circ  T_{\tau}(X) .
\]
\end{theorem}
We call $\mathbf{BZ}_{\tau}$ to be a generalized Bernstein-Zelevinsky functor, reflecting our original viewpoint from $p$-adic groups in \cite{CS19, Ch22+}. One can then translate recent results in \cite{Ch22+, Ch22+c} to $\mathcal{HC}_n$, see Section \ref{s bz derivatives applications}. It is an interesting problem to extend Theorem \ref{thm bz gamma tensor} to constructions of other classical groups such as \cite{CT11} and \cite{Ca22}.

Another problem of our interest is the branching laws, whose Lefschetz principle is also investigated in \cite{Ch23}. However,  the branching law is dealt with in the category of Casselman-Wallach representations and irreducible representations could become non-admissible after restriction. This makes harder for using the functor $\Gamma_{n,m}$ to study branching law directly. On the other hand, as seen in this article, results on parabolic inductions and Jacquet functors in \cite{Ch22+b, Ch22+c, Ch22+d} can be transferred in some form and this opens up possible applications to some branching problems, see e.g. Remarks \ref{rmk tensor product} and \ref{rmk deduce branching from jacquet}.



\subsection{Acknowledgements}
The authors would like to thank the referee for very helpful comments and suggesting the contents in Section \ref{subsec nonunit}. This project is supported in part by the Research Grants Council of the Hong Kong Special Administrative Region, China (Project No: 	17305223) and  the National Natural Science Foundation of China (Project No. 12322120, 12371033).

\section{Preliminaries} \label{s prelim}

In this section, we introduce some basic notations. We also review the representation theory for $\mathrm{GL}_n(\mathbb C)$ and $\mathbb H_m$, which also demonstrates some similarities between their representation theories.

\subsection{Complex groups} \label{ss notation cplx gp}
Let $G = \mathrm{GL}_n(\mathbb{C})$ be a complex Lie group treated as a real group, with a maximal compact subgroup $K = U(n)$ consisting of unitary matrices i.e.
\[  U(n):= \left\{ E \in G: \overline{E}^tE=I_n \right\},
\]
where $\overline{E}$ is the complex conjugation of $E$ and $E^t$ is the transpose of $E$. Let $B$ be the Borel subgroup of upper triangular matrices in $\mathrm{GL}_n(\mathbb C)$. Let $H$ be the subgroup of diagonal matrices in $\mathrm{GL}_n(\mathbb C)$. Let $H=TA$ be the Cartan decomposition of $H$ so that $T \cong (S^1)^{\times n}$ and $A \cong (\mathbb R_{>0})^{\times n}$. Write $\frg_0= \mathfrak{gl}_n(\mathbb{C})$, $\frk_0$, $\frh_0$, $\frt_0$ and $\fra_0$ as their Lie algebras, and remove the subscripts for their complexifications. (Those notions for Lie algebras depend on $n$, but it will be clear from the context.) The group $S_n$ acts on $\mathfrak h_0$ by permuting the coordinates. We use $j$ to denote the action of $\sqrt{-1}$ coming from the complexification (for instance, one has
$\mathfrak{g} = \{P+jQ\ |\ P,Q \in \mathfrak{g}_0\}$).

Following \cite{D} and \cite[Section 7.1]{Vo81}, we make the following identifications for $\frg$ and its subalgebras:
Define $\phi^L, \phi^R: \frg_0 \to \frg$ by:
\begin{align} \label{eqn isomorphism 1}
\phi^L(E) := \frac{1}{2}(E-jiE),\quad 
\phi^R(E) := \frac{1}{2}(\overline{E}+ji\overline{E}),
\end{align}
where $i$ is $\sqrt{-1}I_n$ in
$\frg_0 $. Then one can easily check there is an isomorphism:
\begin{align} \label{eqn isomorphism as comple lie algebra}
\frg_0\oplus \frg_0 \cong  \frg .
\end{align}
given by $(E,E') \mapsto \phi^L(E)+\phi^R(E')$.
Using the isomorphism, all elements in $\frg$ will be written as $(E,E') \in \frg_0 \oplus \frg_0$ from now on. In particular, for $A, B \in \frk_0$, 
$A+jB = \phi^L(A+iB) + \phi^R(\overline{A-iB})\in \frk$ corresponds to
$$(A+iB, \overline{A-iB}) = (A+iB, \overline{A}+i\overline{B}) = (A+iB, -A^t-iB^t) = (A+iB,-(A+iB)^t),$$
since $A^t = -\overline{A}$ and 
$B^t = -\overline{B}$. So we have the identifications
\begin{equation}\label{identification}
 \quad \frk \cong \{(E, -E^t): E \in\frg_0\},\quad \quad \frt\cong\{(H, -H): H\in \frh_0\},\quad \quad \fra\cong\{(H, H): H\in \frh_0\}.
\end{equation}







\subsection{Representation theory of $\mathrm{GL}_n(\mathbb C)$} \label{sec-ps}
Recall that $\mathcal{HC}_n$ is the category of Harish-Chandra modules (a.k.a. $(\mathfrak g, K)$-modules) of $\mathrm{GL}_n(\mathbb C)$ in the sense of \cite[Definition 0.3.8]{Vo81}. Note that for any Harish-Chandra module $X$ in $\mathcal{HC}_n$, the actions of $\phi^L(E)$ and $\phi^L(iE)$ differ by the scalar $\sqrt{-1}$. The same holds if we replace $\phi^L$ by $\phi^R$.

Let $U(\mathfrak g_0 \oplus \mathfrak g_0)$ be the universal enveloping algebra of $\mathfrak g_0\oplus \mathfrak g_0$ over $\mathbb C$. Using the above discussion, we can and shall regard $X$ as a $U(\mathfrak g_0\oplus \mathfrak g_0)$-module determined by the action:
\[   (E_1,E_2). v= (\phi^{L}(E_1)+\phi^R(E_2)).v
\]
for $v \in X$ and $(E_1, E_2) \in \mathfrak g_0\oplus \mathfrak g_0$.




For $n_1+n_2=n$, let $G=\mathrm{GL}_{n_1+n_2}(\mathbb C)$ and let $P=P_{n_1,n_2}$ be the parabolic subgroup containing matrices of the form 
\begin{equation} \label{eq-n1n2}  \begin{pmatrix} g_1 & * \\ & g_2 \end{pmatrix}
\end{equation}
for $g_i \in \mathrm{GL}_{n_i}(\mathbb C)$ $(i=1,2)$. For a Casselman-Wallach representation $\tau$ of $\mathrm{GL}_{n_1}(\mathbb C)\times \mathrm{GL}_{n_2}(\mathbb C)$-representation, $\tau$ is inflated to a $P$-representation and let $\mathrm{Ind}_P^{G}(\tau)$ be the space of smooth functions from $G$ to $\tau$ satisfying: for $p \in P$ and $k \in K$,
\[    f(pk) =\delta^{1/2}(p) (p\cdot f(k)),
\]
where $\delta$ is the modular character of $P$. The $G$-action is given by the right regular action. Here a Casselman-Wallach representation of a reductive Lie group means a smooth admissible Fr\'eceht representation of moderate growth. For the relation between Casselman-Wallach representations and Harish-Chandra modules, one sees discussions \cite[Sections 8-10]{Cas89} or \cite{BK14}.

For a Harish-Chandra module $\tau$ of $\mathrm{GL}_{n_1}(\mathbb C) \times \mathrm{GL}_{n_2}(\mathbb C)$, by abuse of notations, we also write $\mathrm{Ind}_P^G(\tau)$ to be $\mathrm{Ind}_P^G(\widetilde{\tau})$ for some smooth Fr\'echet globalization $\widetilde{\tau}$ of $\tau$ inflated to a $P_{n_1,n_2}$-representation.

For $Y_1$ in $\mathcal{HC}_{n_1}$ and $Y_2$ in $\mathcal{HC}_{n_2}$, define the real parabolic induction as:
\[   Y_1 \times Y_2 = \mathrm{Ind}_{P}^{G} (Y_1 \boxtimes Y_2)_{K-{\bf finite}},
\]
which is regarded as a representation in $\mathcal{HC}_{n_1+n_2}$ and the action from $G$ is the right translation. It follows from \cite[Page 179]{Knp86} that the product is an associative operation. 

\begin{definition} \label{def ps hc}
For any $a, b \in \mathbb{C}$ satisfying $a-b \in \mathbb{Z}$, $\chi_{a, b}$ is a $\mathrm{GL}_1(\mathbb{C})$-character given by 
\begin{equation} \label{eqn character gl1}   \chi_{a, b}(z)=z^{a}\bar{z}^{b}.
\end{equation}
Let  $\lambda_L = (\lambda_{L,1}, \dots, \lambda_{L,n})$, $\lambda_R = (\lambda_{R,1}, \dots, \lambda_{R,n}) \in \mathfrak{h}_0^* \cong \mathbb{C}^n$
be such that
$\lambda_{L,i}-\lambda_{R,i} \in \mathbb{Z}$ for all $i$. The {\bf principal series (representation)} $X(\lambda_L,\lambda_R)$ is defined as:
\begin{equation} \label{eq-psinduce} X(\lambda_L,\lambda_R) = \chi_{\lambda_{L,1},\lambda_{R,1}} \times \chi_{\lambda_{L,2},\lambda_{R,2}}  \times \dots \times \chi_{\lambda_{L,n},\lambda_{R,n}}.
\end{equation}
Write $\mu := \lambda_L - \lambda_R$ and $\nu := \lambda_L + \lambda_R$. Using the identification \eqref{identification}, 
it can also be written as:
\[X(\lambda_L,\lambda_R) = \mathrm{Ind}_{B}^G(\mathbb{C}_{\mu} \boxtimes \mathbb{C}_{\nu} \boxtimes 1)_{K-{\bf finite}},\] where we write $B = TAN$ with $N$ being the subgroup of unipotent upper triangular matrices. 
 
The $K$-type with extremal (but not necessarily dominant) weight $\mu$ occurs exactly once in $X(\lambda_L, \lambda_R)$. Let $J(\lambda_L, \lambda_R)$ be the unique irreducible subquotient of $X(\lambda_L, \lambda_R)$ containing this $K$-type (see \cite{PRV67}).
\end{definition}


\smallskip
The principal series plays a more prominent role in the representation theory of complex Lie groups since it constructs all irreducible Harish-Chandra modules in a nice manner:

\begin{theorem}[\cite{Zh}] \label{thm-Zh}
Retain the above notations. Then the following statements hold:
\begin{itemize}
\item[(a)] Every irreducible Harish-Chandra module is of the form $J(\lambda_L,\lambda_R)$.
\item[(b)] Two such modules $J(\lambda_L,\lambda_R)$ and
$J(\lambda_L^{\prime},\lambda_R^{\prime})$ are equivalent if and
only if there exists $w\in S_n$ such that
$w\lambda_L=\lambda_L^{\prime}$ and $w\lambda_R=\lambda_R^{\prime}$.
\end{itemize}
\end{theorem}

We say $X(\lambda_L,\lambda_R)$ is a {\bf standard module} if $\mathrm{Re}(\nu) = \mathrm{Re}(\lambda_L + \lambda_R )$ is dominant i.e. 
\[    \mathrm{Re}(\lambda_{L,1}+\lambda_{R,1})\geq \ldots \geq \mathrm{Re}(\lambda_{L,n}+\lambda_{R,n}) .
\]
It is well-known (see e.g. \cite[Th\'eor\`eme I.4.2]{D}, \cite[Theorem 8.54]{Knp86}) that if $X(\lambda_L,\lambda_R)$ is standard, then it has a unique maximal proper submodule, so that $J(\lambda_L,\lambda_R)$ appears as a unique quotient of $X(\lambda_L,\lambda_R)$. The notion of standard representations is more convenient (than principal series) when later we have to compare with the modules of graded Hecke algebras (defined in next section).


\subsection{Graded Hecke algebras} \label{ss gaha}

\begin{definition} \label{def gaha type a}
The graded Hecke algebra $\mathbb H_m$ of type $A$ is the associative unital algebra over $\mathbb C$ with generators $y_1, \ldots, y_m$ and $s_1, \ldots, s_{m-1}$ such that 
\begin{itemize}
    \item $y_iy_j=y_jy_i$ for any $i,j$;
    \item $s_i^2=1$ for all $i$;
    \item $s_is_j=s_js_i$ for $|i-j|>1$;
    \item $s_is_{i+1}s_i=s_{i+1}s_is_{i+1}$;
    \item $s_iy_i-y_{i+1}s_i=1$;
    \item $s_iy_j-y_js_i=0$ if $j\neq i, i+1$.
\end{itemize}
\end{definition}
Let $S_m$ be the symmetric group permuting $m$ elements. Note that $s_1, \ldots, s_{m-1}$ generate the group algebra $\mathbb C[S_m]$. Thus, for $w \in S_m$, we also regard as an element in $\mathbb H_m$ via the natural embedding from $\mathbb C[S_m]$ to $\mathbb H_m$.

There is a natural embedding of $\mathbb H_{m_1}\otimes \mathbb H_{m_2}$ to $\mathbb H_{m_1+m_2}$ via the maps: for $i=1,\ldots, m_1$ and $j=1,\ldots, m_2$,
\[     y_i\otimes 1 \mapsto y_i , \quad 1 \otimes y_j \mapsto y_{m_1+j} ,\]
for $i=1, \ldots, m_1-1$ and $j=1, \ldots, m_2-1$,
\[     s_i \otimes 1 \mapsto s_i, \quad 1 \otimes s_j \mapsto  s_{m_1+j} .\]
For $\mathbb H_{m_1}$-module $\pi_1$ and $\mathbb H_{m_2}$-module $\pi_2$, we write 
\[  \pi_1 \times \pi_2 := \mathbb H_{m_1+m_2} \otimes_{\mathbb H_{m_1}\otimes \mathbb H_{m_2}} (\pi_1\boxtimes \pi_2).
\]
The associativity of this product follows from the associativity of tensor products and the standard fact that the tensoring $\mathbb A\otimes_{\mathbb A}$ is the identity functor for any algebra with an unit.

\subsection{Representation theory of $\mathbb H_m$}

The classification for irreducible $\mathbb H_m$-modules is known for long time, see e.g. \cite{Ze80, Ro86}. Let $\mathcal H_m$ be the category of finite-dimensional $\mathbb H_m$-modules.

A segment is of the form $[a,b]$ for some $a,b \in \mathbb C$ with $b-a\geq 0$. We shall consider a segment $[a,b]$ to be a set $\left\{ a, \ldots, b\right\}$. A multisegment is a multiset of non-empty segments. For convenience, set $[a,a-1]=\emptyset$, and a segment can be an empty set and a multisegment can also be an empty set.

Two segments $\Delta_1$ and $\Delta_2$ are said to be linked if $\Delta_1\cup \Delta_2$ is still  a segment, and $\Delta_1\not\subset \Delta_2$ and $\Delta_2\not\subset \Delta_1$. For two segments $\Delta_1=[a_1,b_1], \Delta_2=[a_2,b_2]$, we write $\Delta_1< \Delta_2$ if $\Delta_1$ and $\Delta_2$ are linked and $a_1<a_2$. 

For $c \in \mathbb C$, define $\chi_c$ to be a character on $\mathbb H_1=\mathbb C[y]$ such that $\psi_c(y)=c$. For each segment $\Delta=[a,b]$, define $\mathrm{St}(\Delta)$ to be the unique simple quotient of 
\[   \psi_{a} \times \psi_{a+1} \times \ldots \times \psi_{b}
\]
Moreover, $\mathrm{St}(\Delta)$ is one-dimensional and is the sign representation as a $S_n$-representation. A standard property on parabolic inductions is that: if $\Delta_1$ and $\Delta_2$ are not linked, then
\begin{equation} \label{eqn unlinked commute} \mathrm{St}(\Delta_1) \times \mathrm{St}(\Delta_2) \cong \mathrm{St}(\Delta_2) \times \mathrm{St}(\Delta_1) .
\end{equation}

For a multisegment $\mathfrak m=\left\{ \Delta_1, \ldots, \Delta_k\right\}$, we label the segments such that 
\begin{equation}  \label{eqn prec ordering}
\Delta_1 \not< \ldots \not< \Delta_k .
\end{equation}
Then, let
\[  \lambda(\mathfrak m):= \mathrm{St}(\Delta_1)\times \ldots \times \mathrm{St}(\Delta_k) . \]
This is called a {\bf standard module}, and the unique simple quotient of $\lambda(\mathfrak m)$ is denoted by $\mathrm{St}(\mathfrak m)$. It follows from (\ref{eqn unlinked commute}) that $\lambda(\mathfrak m)$ is independent of a choice of an order of segments in $\mathfrak m$. In particular, for writing each segment in $\mathfrak m$ as $\Delta_i=[a_i, b_i]$, we can choose an ordering satisfying (\ref{eqn prec ordering}) by fixing:
\[  \mathrm{Re}(a_1) \geq \mathrm{Re}(a_2) \geq \ldots \geq \mathrm{Re}(a_n) .
\]
Standard modules construct all irreducible $\mathbb H_m$-modules in the following sense:

\begin{theorem} \cite{Ze80} \label{thm classify irred hecke}
\begin{enumerate}
\item For any irreducible $\mathbb H_m$-module $\pi$, there is a multisegment $\mathfrak m$ such that $\pi \cong \mathrm{St}(\mathfrak m)$.
\item For two multisegments $\mathfrak m_1, \mathfrak m_2$, if $\mathrm{St}(\mathfrak m_1)\cong \mathrm{St}(\mathfrak m_2)$, then $\mathfrak m_1=\mathfrak m_2$.
\end{enumerate}
\end{theorem}

\begin{remark} \label{rmk construction of irreducible from S_m types}
One may notice there is an alternate way to define irreducible $\mathbb H_m$-modules in \cite{Ro85} by using a uniqueness of some $S_m$-types. This supplements a parallel story (i.e. Lefschetz principle) to using $K$-types in defining irreducible Harish-Chandra modules in Definition \ref{def ps hc}. 
\end{remark}

\section{Arakawa-Suzuki type functors} \label{s as functor}

\subsection{`Standard' representations} \label{ss std repn}

The definition of our Arakawa-Suzuki type functor requires a choice of a `standard representation' $V$ of $\mathrm{GL}_n(\mathbb C)$. For reasons that become obvious later in this paper (see Section \ref{ss standard module dimension} below), we make the following choice of the standard representation:
\begin{lemma} \label{lem std repn}
Let $V \cong \mathbb C^n$ be the conjugate standard representation of $\mathrm{GL}_n(\mathbb C)$ i.e. $g \in \mathrm{GL}_n(\mathbb C)$ acting on $V$ by the matrix multiplication of $\overline{g}$. Then $V=J(\rho,e_1+ \rho)$, where $\rho$ is half the sum of positive roots in the root system determined by $H$ and $B$, and $\{e_i \}_{1 \leq i \leq n}$ is the standard basis of $\mathbb{C}^n$.

\end{lemma}

\begin{proof}
Let $\xi:\mathfrak{g}\to\mathfrak{gl}(V)$ be the (complexified) derivative of the conjugate standard representation, that is, for all $X, Y \in \mathfrak{g}_0$, $\xi(X+jY):=\overline{X}+i\overline{Y} \in \mathfrak{gl}(V)$. 
Under the isomorphism $\mathfrak{g}_0 \oplus \mathfrak{g}_0 \cong \mathfrak{g}$ given in \eqref{eqn isomorphism as comple lie algebra}, one has
\begin{align*}
    \xi(P,0)&=\frac{1}{2}\xi(P-jiP) = \frac{1}{2}(\overline{P}-i(-i\overline{P})) = 0 \\
    \xi(0,Q)&=\frac{1}{2}\xi(\overline{Q}+ji\overline{Q}) = \frac{1}{2}(Q+i(-iQ)) = Q
\end{align*}
So $V$ is trivial (with infinitesimal character $\rho$) on the first copy of $\mathfrak{g}_0$, 
and is the standard representation (with infinitesimal character $e_1 + \rho$) on the second copy of $\mathfrak{g}_0$. 
\end{proof}

\subsection{Casimir element}

Let $E_{ij} \in \mathfrak{g}_0$ be the $n \times n$ matrix with $1$ on the $(i,j)$-entry and $0$ on the other entries. For $0\leq k<l \leq m$, define the Casimir element:
\[ \Omega_{kl} := \sum_{1 \leq i, j \leq n} 1^{\otimes k} \otimes  E_{ij} \otimes 1^{\otimes l-k-1} \otimes E_{ji} \otimes 1^{m-l}  \in U(\frg_0)^{\otimes (m+1)}  . \]
Here $1$ is the unit in $U(\mathfrak g_0)$.

Note that $\left\{ E_{ji} \right\}$ is a dual basis for $\left\{ E_{ij} \right\}$ under the pairing:
\[   (E, E') \mapsto \mathrm{tr}(EE'{}^t) 
\]
and so for any $g \in \mathrm{GL}_n(\mathbb C)$, $\Omega_{kl}$ is invariant under the adjoint action of $\mathrm{GL}_n(\mathbb{C})$:
\begin{align} \label{eqn invariant casimir}
\mathrm{Ad}(g)(\Omega_{kl}):= \sum_{1\leq i,j \leq n} 1^{\otimes k}\otimes gE_{ij}g^{-1}\otimes 1^{\otimes (l-k-1)} gE_{ji}g^{-1} \otimes 1^{\otimes (m-l)} =\Omega_{kl} .
\end{align}

\subsection{Arakawa-Suzuki realization of $\mathbb H_m$} \label{sec-as}

\begin{theorem} \cite{AS98} 
There is an injective algebra (over $\mathbb C$) homomorphism $\Theta: \mathbb{H}_m \longrightarrow U(\frg_0)^{\otimes (m+1)}$ defined by:
$$\begin{cases}
\Theta(s_i) := -\Omega_{i,i+1},  &1 \leq i < m;\\
\Theta(y_{l}) := \sum_{0 \leq x < l} \Omega_{x,l} + \frac{n}{2}(1^{\otimes (m+1)}), &1 \leq l \leq m.
\end{cases}$$
We shall write $\Theta_{\mathbb H_m}$ for $\Theta$ if we have to specify the underlying graded Hecke algebra.
\end{theorem}

Recall that $V$ is defined in Lemma \ref{lem std repn}. Let $X$ be in $\mathcal{HC}_n$. Define the algebra homomorphism $\Lambda: U(\mathfrak g_0)^{\otimes (m+1)}\rightarrow \mathrm{End}(X \otimes V^{\otimes m})$ determined by:
\begin{align} \label{eqn action from tensors}
     \Lambda(1^{\otimes k} \otimes E \otimes 1^{\otimes (m-k)}).(\mathbf x \otimes v_1 \otimes \ldots \otimes v_m) = \mathbf x \otimes v_1 \otimes \ldots \otimes (0,E).v_k\otimes \ldots \otimes v_n ,
\end{align}
where $(0,E)$ is the corresponding complexified Lie algebra action on $V$. (When $k=0$, the copy $(0,E)$ acts on $\mathbf x$.)

\begin{definition} \label{def-as} (c.f. \cite{AS98, CT12})
Let $X$ be in $\mathcal{HC}_n$, and $V$ be the conjugate standard representation as in Lemma \ref{lem std repn}. We define the Arakawa-Suzuki type functor:
$$\Gamma_{n,m}: \mathcal{HC}_n \longrightarrow \mathcal{H}_m$$
to be the exact covariant functor given by
\begin{equation} \label{eq-as}
\Gamma_{n,m}(X) := \mathrm{Hom}_K(\mathrm{triv}, X \otimes V^{\otimes m}).
\end{equation}
where the $\mathbb{H}_m$-action on $\Gamma_{n,m}(X)$ is given by the map 
\[ \Lambda\circ \Theta: \mathbb H_m \rightarrow \mathrm{End}_{\mathbb C}(X\otimes V^{\otimes m}).
\]
\end{definition}

We remark that it follows from (\ref{eqn invariant casimir}) that the action of $\mathbb H_m$ in Definition \ref{def-as} is well-defined i.e. $(\Lambda \circ \Theta(h))(\Gamma_{n,m}(X)) \subset \Gamma_{n,m}(X)$ for any $h \in \mathbb H_m$.

An interesting feature of our functor $\Gamma_{n,m}$ is that the action of $\mathbb H_m$ is only via one of the copies on $\mathfrak{gl}_n(\mathbb C)_{\mathbb C} \cong \mathfrak{gl}_n(\mathbb C)\oplus \mathfrak{gl}_n(\mathbb C)$. However, from our choice of a `standard' representation $V$ in Section \ref{ss std repn}, $\phi^L(\mathfrak{gl}_n(\mathbb C))$ always act trivially. This explains why it is not important to incorporate another copy of $\mathbb H_m$ for our functor.

\subsection{Another formulation} \label{sec-another}
Define $\bar{V}$ to be the contragredient of the standard representation of $\mathrm{GL}_n(\mathbb C)$. Then one can check as in Lemma \ref{lem std repn} that $\bar{V}=J(-e_n+\rho, \rho)$.


We define $\zeta: \mathfrak g_0 \oplus \mathfrak g_0 \rightarrow \mathfrak g_0 \oplus \mathfrak g_0$ by:
\[    \zeta(E,E') =  (-E', -E) .
\]
This comes from the involution $g \mapsto \bar{g}^{-t}$ and induces an auto-equivalence of categories on $\mathcal{HC}_n$, still denoted by $\zeta$. Define $\bar{\zeta}: \mathfrak h_0\oplus \mathfrak h_0\rightarrow \mathfrak h_0 \oplus \mathfrak h_0$ by $\bar{\zeta}(\lambda_L, \lambda_R)=(-\lambda_R, -\lambda_L)$ so that  
\[ \zeta(J(\lambda_L, \lambda_R)) =J(\bar{\zeta}(\lambda_L, \lambda_R)) .
\]
 We also note that $\zeta(V)=\bar{V}$.
 
Define 
\[  \bar{\Lambda}: U(\mathfrak g_0)^{\otimes (m+1)} \rightarrow \mathrm{End}_{\mathbb C}(X\otimes \bar{V}^{\otimes m})
\]
given by 
\[   \bar{\Lambda}(1^{\otimes k}\otimes E \otimes 1^{\otimes (m-k)})(\mathbf x\otimes v_1\otimes \ldots \otimes v_m)=\mathbf x\otimes v_1\otimes \ldots \otimes (E,0)\cdot v_k\otimes \ldots \otimes v_n.
\]
and define the dual functor $\bar{\Gamma}_{n,m}: \mathcal{HC}_n \rightarrow \mathcal H_m$ given by \[\bar{\Gamma}_{n,m}(X)= \mathrm{Hom}_K(\mathrm{triv}, X \otimes \bar{V}^{\otimes m}) \]
with the action of $\mathbb H_m$ given by $\bar{\Lambda} \circ \Theta$. As a result, we have:
\[ \bar{\Gamma}_{n,m}\circ \zeta \cong \Gamma_{n,m} .
\]

\subsection{Other Arakawa-Suzuki type functor involving the Bernstein-Gelfand functor}


We discuss another way to define a Arakawa-Suzuki type functor for $\mathcal{HC}_n$. We first recall a functor due to Bernstein-Gelfand \cite{BG80} connecting to the BGG category $\mathcal O_n$ for $\mathfrak{gl}_n(\mathbb C)$.  To facilitate the setup in \cite{BG80},  we shall identify, $\mathcal{HC}_n$ with the category $\mathcal{HC}^b_n$ of Harish-Chandra $U(\mathfrak g_0)$-bimodules  i.e. the category of $(U(\mathfrak g_0), U(\mathfrak g_0))$-modules  with a local $\mathfrak{k}$-finiteness condition (in the sense of \cite[Section 5.1]{BG80}). See \cite[Appendix II]{BG80} for more details on such identification. We shall now define another functor from some subcategory of $\mathcal{HC}^b_n$ to $\mathcal H_m$.

Let $\chi \in \mathfrak{h}_0^*$ be a dominant character. Then $\chi$ determines the central character $\chi^*:Z(\mathfrak{g}_0) \to \mathbb{C}$ given by
$\chi^*(z):=(\chi-\rho)(\mathrm{pr}(z))$, where $\mathrm{pr}: U(\mathfrak g_0) \rightarrow U(\mathfrak h_0)$ by setting other PBW monomials zero.
Let $\mathcal{HC}_n^{R,\chi}$ be the subcategory of $\mathcal{HC}_n^b$ such that all $X \in \mathcal{HC}_n^{R,\chi}$ is annihilated by $\phi^R(\ker(\chi^*))$. For instance, for any $\xi \in \mathfrak{h}_0^*$ with $\xi-\chi$ to be integral, all subquotients of the principal series representation $X(\xi,\chi)$ are in $\mathcal{HC}_n^{R,\chi}$.


Define the functor $T_{\chi}: \mathcal{HC}_n^{R,\chi} \to \mathcal{O}_n$
by 
\[T_{\chi}(X) := X \otimes_{U(\mathfrak{g}_0)} M(\chi), \]
where $M(\chi)$ is the Verma module with highest weight $\chi-\rho$. 
We also denote by $F_{\chi'}$ the Arakawa-Suzuki functor in \cite{AS98}. Then we can obtain a functor $F_{\chi'} \circ T_{\chi}$ from $\mathcal{HC}^{R,\chi}_n$ to $\mathcal H_m$. It is an interesting problem to compare $F_{\chi'}\circ T_{\chi}$ with $\bar{\Gamma}_{n,m}$. For instance, if $\chi$ is dominant and regular, then for any subquotient $Y$ of $X(-\xi,-\chi)$ in $\mathcal{HC}_n^{R,\chi}$, one expects that $F_{\chi}\circ T_{\chi}(Y)$ and $ \bar{\Gamma}_{n,m}(Y)$ are almost isomorphic. Here we use 'almost' because indeed one should modify the functor by choosing another standard representation not covered in Sections \ref{ss std repn} and \ref{sec-another}.

On the other hand, our functor is defined in $\mathcal{HC}_n$ and is applicable for some wider potential applications (see Section \ref{ss remark on higher structures}).

\section{Computations on the actions of $\mathbb H_m$ on induced modules} \label{s compute hm induced mod}

We shall compute some actions of $\mathbb H_n$ on $\Gamma_{n,m}(X)$ for some parabolically induced modules, which will be used to prove Theorem \ref{thm isomophic of parabolic induction} in the next section.

\subsection{Some identifications} \label{ss identifications}

We identify some spaces, which will be useful in defining some generators for $\mathbb H_m$ under $\Gamma_{n,m}$ (see Section \ref{ss special subspace}). For a parabolic subgroup $P$ of $\mathrm{GL}_n(\mathbb C)$ containing $B$ and a Harish-Chandra module $\tau$ of $P$,  
\begin{itemize}
    \item $\mathrm{Ind}_P^G (\tau)\otimes V^{\otimes m} \cong \mathrm{Ind}_P^G (\tau \otimes V^{\otimes m})$, as $\mathrm{GL}_n(\mathbb C)$-representations, via the map:  for $f \in \mathrm{Ind}_P^G(\tau)$,
    \[  f \otimes v_1\otimes \ldots \otimes v_m \stackrel{\Phi}{\mapsto} (g \mapsto f(g) \otimes g\cdot 
 v_1\otimes \ldots \otimes g\cdot v_m) .
    \]
    This is a bijection, see for example \cite[Theorem 44.1]{Tr06}.
    \item $(\mathrm{Ind}_P^G(\tau))^K  \cong \tau^{K_P}$, 
    where $K_P=K \cap P$. The natural map is given by the restriction and the inverse map is given by: for  $v \in \tau^{K_P}$, define $f \in \mathrm{Ind}_P^G\tau$ as:
    \[   f(pk)=p \cdot v .
    \]
    It is straightforward to check that last map is well-defined. 
\end{itemize}

\subsection{Computation via differentiations}

We use notations in the previous section. For $E \in \mathfrak g_0$ and $f \in \mathrm{Ind}_P^G(\tau)$, define:
\[  (E.f)(g)=\left. \frac{d}{ds}f(g \cdot \mathrm{exp}({sE})) \right|_{s=0} ,
\]
and $((jE).f)(g)=\sqrt{-1}((E.f)(g))$. (Here $\mathrm{exp}$ is the usual exponential map.)

Again we identify $\mathfrak g_0 \oplus \mathfrak g_0$ with $\mathfrak g$ and so for $(E,E') \in \mathfrak g_0\oplus \mathfrak g_0$, we mean:
\[  ((E,E').f)(g)=(\phi^L(E).f)(g)+(\phi^R(E').f)(g) .\]
We shall frequently use these formulas in our computations in Sections \ref{ss basic formulas} to \ref{ss compute lower part}.

\subsection{Cyclic subspace} \label{ss special subspace}

Let $\tau$ be a Harish-Chandra module of $\mathrm{GL}_{n_1}(\mathbb C)\times \mathrm{GL}_{n_2}(\mathbb C)$. Let $n=n_1+n_2$ and $m=m_1+m_2$. Note that $\tau \otimes V^{\otimes m}$ contains a subspace spanned by vectors of the form:
\[   \mathbf x \otimes v_1\otimes \ldots \otimes v_{m_1} \otimes v_{m_1+1} \otimes \ldots \otimes v_{m_1+m_2}
\]
for $\mathbf x \in \tau$, $v_1, \ldots, v_{m_1} \in \mathbb C^{ n_1}\subset \mathbb C^{n} \cong V $ and $v_{m_1+1}, \ldots, v_{m_1+m_2} \in \mathbb C^{n_2}\subset \mathbb C^{n} \cong V$. Here, the first inclusion by sending $\mathbb C^{n_1}$ to the first $n_1$-coordinates and the second inclusion by sending $\mathbb C^{n_2}$ to the last $n_2$-coordinates of $V$. We shall denote such subspace by $\mathcal W'(\tau, n_1, n_2, m_1, m_2)$.

 Let $P=P_{n_1,n_2}$ be as defined in \eqref{eq-n1n2}. Let $\mathcal W(\tau, n_1, n_2, m_1, m_2)$ be the subspace of $\mathrm{Ind}_P^G(\tau \otimes V^{\otimes (m_1+m_2)})$ containing all functions satisfying the properties:
    \begin{enumerate}
        \item $F(1) \in \mathcal W'(\tau, n_1, n_2, m_1, m_2)$; and
        \item $F(k)=F(1)$ for all $k \in U(n)$.
    \end{enumerate}
Note that (1) and (2) above also imply that $F(1) \in (\tau \otimes V^{\otimes m})^{U(n_1)\times U(n_2)}$ since $P\cap K=U(n_1)\times U(n_2)$.

\subsection{Basic formulas} \label{ss basic formulas}

We shall keep using the notations in Section \ref{ss special subspace} until Section \ref{ss compute lower part}. For each element $E \in \mathfrak g_0$, and for an non-negative integer $k$, we define 
\[  \Delta_k(E) = 1^{\otimes k} \otimes (0,E) \otimes  1^{\otimes (m-k)} \]
and for $k<l$,
\[ \Delta_{kl}(E)  = 1^{\otimes k} \otimes (0,E) \otimes 1^{\otimes (l-k-1)} \otimes (0,E^t) \otimes 1^{\otimes (m-l)} .\]

By abuse of notations, we write: for $F \in \mathrm{Ind}_P^G(\tau \otimes V^{\otimes m})$, 
\[   \Delta_l(E) \cdot F := \Phi(\Delta_l(E) \cdot \Phi^{-1}(F)) , \quad \Delta_{kl}(E) \cdot F := \Phi(\Delta_{kl}(E) \cdot \Phi^{-1}(F)) ,
\]
where $\Phi$ is the isomorphism given in Section \ref{ss identifications} and the $\Delta_l(E)$ action is defined in (\ref{eqn action from tensors}) (via $\Lambda$). We similarly define for $\Delta_{kl}(E)\cdot F$.

For $g \in \mathrm{GL}_n(\mathbb C)$ and $F \in \mathrm{Ind}_P^G(\tau \otimes V^{\otimes m})$, we shall write
\[ (g \cdot F)(h)=F(hg) .\]
For $(E, E') \in \mathfrak g_0\oplus \mathfrak g_0$, we write $(E,E')\cdot F$ to be the induced Lie algebra action.

\begin{lemma} \label{lem basic formula 1}
 Let $F \in \mathrm{Ind}_P^G(\tau \otimes V^{\otimes m})$ with $F(1)=\sum_i \mathbf x_i\otimes v_{i,1}\otimes \ldots \otimes v_{i,m}$. For $r \geq 1$, one has
 \[   (\Delta_r(E) \cdot F)(1)= \sum_i \mathbf x_i \otimes v_{i,1} \otimes \ldots \otimes (0,E)\cdot v_{i,r}\otimes\ldots \otimes v_{i,m} .
 \]
\end{lemma}

\begin{proof}
We consider $f \otimes v_1 \otimes \ldots \otimes v_m \in \mathrm{Ind}_P^G(\tau)\otimes V^{\otimes m}$. Let $F'=\Phi(f\otimes v_1\otimes \ldots \otimes v_m)$. Then 
\begin{align*}
  (\Delta_r(E)\cdot F')(g) &= f(g)\otimes g\cdot v_1\otimes \ldots \otimes  (g\cdot(0,E) \cdot v_r)\otimes \ldots \otimes g\cdot v_m 
  \end{align*}
Now, evaluating at $g=1$, we have:
\begin{align*}
   (\Delta_r(E)\cdot F')(1) &= f(1)\otimes v_1\otimes \ldots \otimes (0,E) \cdot v_r \otimes \ldots \otimes v_m .
\end{align*}
This gives the equality for $F'$. Since $\Phi$ is an isomorphism (see Section \ref{ss identifications}), we can now extend linearly.
\end{proof}

\begin{lemma} \label{lem action on zero term}
 $(\Delta_0(E)\cdot F)(1)=((0,E)\cdot F)(1)-(\Delta_1(E)\cdot F)(1)-\ldots -(\Delta_m(E)\cdot F)(1)$.   
\end{lemma}

\begin{proof}
We pick $f  \in \mathrm{Ind}_P^G(\tau)$ and let  $F=\Phi(f\otimes v_1\otimes \ldots \otimes v_m)$. Set $s$ to be a variable. We have:
\begin{align*}
 &(\mathrm{exp}(s(0,E))\cdot F)(g) \\
= & f(g\cdot \mathrm{exp}(s(0,E))) \otimes g \cdot\mathrm{exp}(s(0,E)) \cdot v_1\otimes \ldots \otimes g \cdot \mathrm{exp}(s(0,E)) \cdot v_m 
\end{align*}
Taking differentiation and then evaluating at $s=0$ yields:
\begin{align*}
    ((0,E)\cdot F)(g)    &= ((0,E)\cdot f)(g) \otimes g\cdot v_1\otimes \ldots \otimes g\cdot v_m  +f(g)\otimes  g\cdot (0,E)\cdot v_1\otimes \ldots \otimes g\cdot  v_m + \\
    &  \quad \ldots +f(g)\otimes g\cdot v_1 \otimes \ldots \otimes g \cdot (0,E)\cdot v_m 
\end{align*}

Note that the first term is equal to $\Phi(\Delta_0(E)\cdot (f\otimes v_1\otimes \ldots \otimes v_m))$. Then evaluating at $g=1$, it is equal to $(\Delta_0(E)\cdot F)(1)$.

Starting from the second term, we need to do evaluations at $g=1$  to see that the corresponding term is equal to:
\[   f(1)\otimes v_1 \otimes \ldots \otimes (0,E) \cdot v_r \otimes \ldots \otimes v_m ,
\]
and so is equal to $(\Delta_r(E)\cdot F)(1)$.

Now, we rearrange the second and higher terms to the LHS, we obtain the formula in the lemma for $F=\Phi(f\otimes v_1\otimes \ldots \otimes v_m)$. But, $\Phi$ is an isomorphism (see Section \ref{ss identifications}), and so we also have the general case.
\end{proof}



\subsection{Computing some actions on torus part}

For simplicity, set $\widetilde{\mathcal W}=\mathcal W(\tau, n_1, n_2, m_1, m_2)$. 


\begin{lemma} \label{lem action of torus upper}
Let $F \in \widetilde{\mathcal W}$. For $y\leq n_1$ and $m_1+1 \leq l$ and any $k$,
\[    (\Delta_{kl}(E_{yy})\cdot F)(1) =0 .
\]
\end{lemma}

\begin{proof}
This follows from Lemma \ref{lem basic formula 1} and the fact that $E_{yy}$ acts by zero on the $l$-th copy of $\mathbb C^{n_2}\subset V$.     
\end{proof}

\begin{lemma} \label{lem action of torus lower}
Let $F \in \widetilde{\mathcal W}$. For $n_1+1\leq y$ and $1 \leq k \leq m_1$ and any $l$ with $k <l$, 
\[   (\Delta_{kl}(E_{yy})\cdot F)(1) =0 .
\]
\end{lemma}
\begin{proof}
This follows from  Lemma \ref{lem basic formula 1} and the fact that $E_{yy}$ acts by zero on the $k$-th copy of $\mathbb C^{n_1}\subset V$.  
\end{proof}

\subsection{Computing some actions on unipotent upper triangular part}

We first consider the action arising from the elements of the form $\begin{pmatrix} 0_{n_1} & * \\ &  0_{n_2} \end{pmatrix}$:

\begin{lemma} \label{lem upper corner triangualr comp}
Let $F \in \widetilde{\mathcal W}$. For $y\leq n_1 <z$, 
\[  (\Delta_{kl}(E_{yz})\cdot F)(1) =0
\]
for all $0\leq k<l$.
\end{lemma}

\begin{proof}
We first assume that $k\geq 1$. We divide into two cases:
\begin{itemize}
    \item Case 1: $k< l \leq m_1$. This follows form that $E_{yz}$ acts by zero on the $k$-th copy of $V$ (which takes the form $\mathbb C^{n_1}\subset V$).
    \item Case 2: $m_1+1 \leq l$ (and $k<l$). This follows from that $E_{yz}^t=E_{zy}$ acts by zero on the $l$-th copy of $V$ (which takes the form $\mathbb C^{n_2}\subset V$).
\end{itemize}

We now assume that $k=0$. Since unipotent elements act as the identity on $\tau$, this implies that $(\Delta_0(E_{yz}).F)(1)=0$ and so $(\Delta_{0l}(E_{yz})\cdot F)(1)=0$.
\end{proof}

We now consider upper triangular matrices of the form:
\[ \begin{pmatrix} * & 0 \\ 0 & 0_{n_2} \end{pmatrix} .
\]

\begin{lemma} \label{lem upper upper corner}
Let $F \in \widetilde{\mathcal W}$.  For $y<z\leq n_1$, and for any $k<l$ with $m_1+1\leq l$,
\begin{align}
  (\Delta_{kl}(E_{yz}) \cdot F)(1) =0 .
\end{align}
\end{lemma}

\begin{proof}
This follows from that $E_{yz}^t=E_{zy}$ acts by zero on the $l$-th copy of $V$.
\end{proof}

We now consider upper triangular matrices of the form:
\[  \begin{pmatrix} 0_{n_1} & 0 \\ 0 & * \end{pmatrix} .\]

\begin{lemma} \label{lem lower upper corner}
Let $F \in \widetilde{\mathcal W}$. For $n_1+1 \leq y<z$, and either
\begin{enumerate}
    \item $0 \leq k<l \leq m_1$, or
    \item $1 \leq k\leq m_1 < l$,
\end{enumerate}  
\begin{align}
    (\Delta_{kl}(E_{yz})\cdot F)(1) =0 .
\end{align}
\end{lemma}

\begin{proof}
 This follows from the fact that $E_{yz}$ acts by zero on the $k$-th copy of $\mathbb C^{n_1}\subset V$ for (2) and on the $l$-th copy of $\mathbb C^{n_1} \subset V$ for (1).  
\end{proof}

\subsection{Computing some actions on lower triangular part} \label{ss compute lower part}




We similarly have the following two cases. The first case is to consider elements of the form:
\[  \begin{pmatrix} * &  \\ & 0_{n_2} \end{pmatrix} .\]

\begin{lemma} \label{lem lower upper actiona }
Let $F \in \widetilde{\mathcal W}$. For any $1\leq z<y \leq n_1$ and any $k<l$ with $m_1+1\leq l$, $(\Delta_{kl}(E_{yz})\cdot F)(1)=0$.
\end{lemma}

\begin{proof}
This follows from that $E_{yz}$ acts by zero on the $l$-th copy of $\mathbb C^{n_2} \subset V$.
\end{proof}

The another case is in the form:
\[  \begin{pmatrix} 0_{n_1} &  \\ & * \end{pmatrix} .
\]

\begin{lemma} \label{lem lowere lower corner}
Let $F \in \widetilde{\mathcal W}$. Suppose we are in one of the following cases:
\begin{enumerate}
    \item $0 \leq k<l \leq m_1$; or
    \item $1 \leq k\leq m_1 < l \leq m_1+m_2$
\end{enumerate}
 Then, for $n_1< y <z$, $(\Delta_{kl}(E_{yz})\cdot F)(1)=0$.
\end{lemma}
\begin{proof}
(2) follows from that $E_{yz}$ acts by zero on the $k$-th copy of $\mathbb C^{n_1}\subset V$ and (1) follows from $E_{zy}$ acts by zero on the $l$-th copy of $\mathbb C^{n_1} \subset V$.
\end{proof}

We consider matrices of the form $\begin{pmatrix} 0_{n_1} & \\ * & 0_{n_2} \end{pmatrix}$.

\begin{lemma} \label{lem lower corner action }
Let $F \in \widetilde{\mathcal W}$. Let $m_1 <l$ and  $z \leq n_1 < y$. Then, 
\begin{align*}
   \sum_{k=0}^{l-1}  (\Delta_{kl}(E_{yz}) \cdot F)(1) 
=  & 0
\end{align*}
\end{lemma}

\begin{proof}
We first consider $(\Delta_{0l}(E_{yz})\cdot F)(1)$. We shall compute from using another expression:
\[ (0,  E_{yz})=(-E_{zy}, E_{yz}) +(E_{zy},0) .
\]

We define $\Delta_k^L(E)$ and $\Delta_{kl}^L(E)$ analogously as $\Delta_k(E)$ and $\Delta_{kl}(E)$ respectively, but using $(E,0)$ instead of $(0,E)$. We still have an analogous formula for Lemma \ref{lem action on zero term} for the left version, and so \[ (\Delta^L_0(E_{zy})\cdot F)(1)=((E_{zy},0)\cdot F)(1)-\sum_{1 \leq k \leq m_1+m_2}(\Delta^L_k(E_{zy})\cdot F)(1).
\]
 Since $(E_{zy},0)$ acts by $0$ on $V$, the latter terms are all zero. Since $E_{zy}$ is an upper triangular matrices, we have that
\[   (( E_{zy},0)\cdot F)(1)=(E_{zy},0) \cdot F(1) ,
\]
The right hand side is the Lie algebra action on $F(1) \in \tau \otimes V^{\otimes (m_1+m_2)}$, and so we also have such term to be zero. In conclusion, we have:
$(\Delta_0^L(E_{zy})\cdot F)(1)=0$ and hence $(\Delta_{0l}^L(E_{zy})\cdot F)(1)=0$.

Now we compute 
\[((\Delta_{0}(E_{yz})\cdot F)(1) = ((\Delta_{0}^L(-E_{zy})+\Delta_{0}(E_{yz}))\cdot F)(1)\] 
as an action for the whole term $(-E_{zy}, E_{yz})$, where the term $\Delta^L_0(-E_{zy})$ acts by zero since $E_{zy}$ is an upper triangular matrix and so $f(\mathrm{exp}(s(E_{zy},0))$ is a constant for any $f \in \mathrm{Ind}_P^G(\tau)$. By using a version of Lemma \ref{lem action on zero term} again, we have:
\[  (\Delta_{0}(E_{yz})\cdot F)(1)= ((-E_{zy},E_{yz})\cdot F)(1)-\sum_{1\leq k\leq m_1+m_2}((-\Delta^L_{k}(E_{zy})+\Delta_k(E_{yz}))\cdot F)(1)
\]
The first term in the RHS is zero since $F \in \widetilde{\mathcal W}$ is $K$-invariant and $(-E_{zy}, E_{yz})$ is in $\mathfrak k$. Furthermore, $\Delta^L_k(E_{zy})\cdot F=0$ since its action on $V$ is trivial. Thus, the above expression reduces to give:
\[   (\Delta_0(E_{yz})\cdot F)(1)=-\sum_{1\leq k \leq m_1+m_2} (\Delta_k(E_{yz})\cdot F)(1) .
\]
But we see the terms are zero for $k \geq m_1+1$ by using Lemma \ref{lem basic formula 1}. Thus, we further have:
\begin{align} \label{eqn reduction ofr lower}
  (\Delta_{0}(E_{yz})\cdot F)(1)=-\sum_{1\leq k\leq m_1} (\Delta_{k}(E_{yz})\cdot F)(1) 
\end{align}
Imposing the action of $(0,E_{zy})$ on the $l$-th copy, we have:
\[  (\Delta_{0l}(E_{yz})\cdot F)(1)=-\sum_{1\leq k\leq m_1} (\Delta_{kl}(E_{yz})\cdot F)(1) .
\]

This gives that:
\[ \sum_{0\leq k \leq m_1} (\Delta_{kl}(E_{yz})\cdot F)(1)=0 .
\]
Again $(\Delta_{kl}(E_{yz})\cdot F)(1)=0$ for $l>k \geq m_1+1$, we have:
\[   \sum_{0\leq k \leq l-1} ((\Delta_{kl}(E_{yz})\cdot F)(1)=0
\]

\end{proof}

Let $e_1, \ldots, e_{n_1+n_2}$ be a standard basis for $V$. Let $\langle , \rangle_V$ be the inner product on $V$ given by: for $a, b \in \mathbb C$,
\[  \langle a e_i,b e_j \rangle_V=a\bar{b} \cdot \delta_{ij} .
\]
We define $\delta_l^i: \tau \otimes V^{\otimes m}\rightarrow \tau \otimes V^{\otimes m}$ determined by:
\[ \delta_l^i(\mathbf x\otimes v_1\otimes \ldots \otimes v_m)=\langle v_l, e_i\rangle_V \cdot \mathbf x \otimes v_1 \otimes \ldots \otimes e_i \otimes \ldots \otimes v_m,
\]
where $e_i$ is in the $l$-th position.

\begin{lemma} \label{lem lower lower action 3}
Let $F \in \widetilde{\mathcal W}$. Let $l \leq m_1$ and $z \leq n_1 < y$. Then,
\[   (\Delta_{0l}(E_{yz}) \cdot F)(1) = -\delta_{l}^z(F(1))  .
\]
\end{lemma}

\begin{proof}
The computation for (\ref{eqn reduction ofr lower}) in the previous lemma again yields:
\[  (\Delta_{0}(E_{yz})\cdot F)(1) =-\sum_{1\leq k \leq m_1} (\Delta_k(E_{yz})\cdot F)(1)
\]
Since $E_{zy}$ acts zero on $\mathbb C^{n_1}\subset V$, we have that $(\Delta_l(E_{zy})\cdot \Delta_k(E_{yz}) \cdot F)(1) = 0$ for $k \neq l$ and hence:
\[   (\Delta_{0l}(E_{yz})\cdot F)(1)=- (\Delta_l(E_{zy})\cdot \Delta_l(E_{yz})\cdot F)(1) .
\]
By using Lemma \ref{lem basic formula 1} twice and using $E_{zy}E_{yz}v=E_{zz}v$, we then have 
\[  (\Delta_{0l}(E_{yz})\cdot F)(1)=- \delta^z_l(F(1)) .
\]

\end{proof}

\section{Parabolic induction under Arakawa-Suzuki type functor} \label{s parabolic as functor}

The main goal of this section is to prove Theorem \ref{thm isomophic of parabolic induction}. We will first illustrate how the dimensions of parabolically induced modules behave under $\Gamma_{n,m}$ in Section \ref{ss standard module dimension}, which only involves some simpler dimension formulas. Working the full version of Theorem \ref{thm isomophic of parabolic induction} requires a more substantial analysis on the module structure. 

\subsection{Height of parameters and representations}


\begin{definition} \label{d-height}
Let $(\lambda_L,\lambda_R) \in \mathfrak h_0^*\times \mathfrak h_0^*$ with $\lambda_L-\lambda_R$ to be integral. Write $\mu=\lambda_L-\lambda_R=(\mu_1, \ldots, \mu_n) \in \mathfrak h_0^*$. The height of $(\lambda_L,\lambda_R)$ is defined as
$$\mathrm{ht}((\lambda_L,\lambda_R)) := \begin{cases}
\sum_{i=1}^n \mu_i & \text{if all } \mu_i \geq 0; \\
-\infty & \text{otherwise}.
\end{cases}$$
We also define $\mathrm{ht}(X(\lambda_L, \lambda_R)):=\mathrm{ht}(\lambda_L, \lambda_R)$ and $\mathrm{ht}(J(\lambda_L, \lambda_R)):=\mathrm{ht}(\lambda_L, \lambda_R)$.
\end{definition}

The height will be more important later in picking right choices of Arakawa-Suzuki type functors  in Theorems \ref{thm-std} and \ref{thm irreducibility of functor}.

\subsection{Dimension under $\Gamma_{n,m}$} \label{ss standard module dimension}

Before working on funtoriality of parabolic induction under $\Gamma_{n,m}$, it may be educative to illustrate some simpler nature of parabolic inductions under $\Gamma_{n,m}$.

We consider a principal series $X=X(\lambda_L, \lambda_R)$. By the identifications in Section \ref{ss identifications}, there is a natural isomorphism:
\[  \Gamma_{n,m}(X(\lambda_L, \lambda_R)) \cong (\mathbb{C}_{(\lambda_L,\lambda_R)} \otimes V^{\otimes m})^T .
\]
Note that 
$\mathbb C_{(\lambda_L,\lambda_R)}|_T = \mathbb C_{\lambda_L-\lambda_R} = \mathbb C_{\mu}$, and $V|_K = J(\rho,e_1+\rho)|_K \cong F_{e_1}^* = F_{-e_n}$ as $K$-modules, where $F_{\xi}$ as the irreducible $K$-module of highest weight $\xi$, and $F_{\xi}^*$ as its contragredient representation (see Equation \eqref{eq-finited} in the Appendix for details on finite dimensional representations of $\mathrm{GL}_n(\mathbb C)$). In particular, $V|_T = \bigoplus_{i=1}^n \mathbb C_{-e_i}$ are the $T$-weights of $V$.


It then follows from a direct computation that if $\mathrm{ht}(X(\lambda_L, \lambda_R)) \neq -\infty$, then 
\begin{align} \label{eqn dimension as functor}
    \dim(\Gamma_{n,m}(X(\lambda_L,\lambda_R))) =\left\{ \begin{array}{ll} \frac{m!}{\mu_1! \mu_2! \dots \mu_n!} &  m=\mathrm{ht}(X(\lambda_L, \lambda_R)) \\
                0        &  \mbox{ otherwise } \end{array}  \right. 
\end{align}
If $\mathrm{ht}(X(\lambda_L, \lambda_R))=-\infty$, then for all $m$,
\[  \Gamma_{n,m}(X(\lambda_L, \lambda_R)) =0 .
\]

\begin{lemma} \label{lem product hecke algebra}
Let $\pi_1$ and $\pi_2$ in $\mathcal H_{m_1}$ and $\mathcal H_{m_2}$ respectively. We have:
\begin{align} \label{eqn product hecke algebra}
       \mathrm{dim}(\pi_1\times \pi_2) =\frac{\mathrm{dim}(\pi_1) \cdot \mathrm{dim}(\pi_2)\cdot (m_1+m_2)!}{m_1! \cdot m_2!} ,
\end{align}
where $\times$ is the induction defined in Section \ref{def gaha type a}.
\end{lemma}

\begin{proof}
This follows from the well-known fact that $t_w\otimes v$ form a basis for $\mathbb H_{m_1+m_2}\otimes_{\mathbb H_{m_1}\otimes \mathbb H_{m_2}}(\pi_1\boxtimes \pi_2)$, where $w$ runs for all minimal representatives in $(S_{m_1}\times S_{m_2})\setminus S_{m_1+m_2}$ and $v$ runs for a fixed basis for $\pi_1\boxtimes \pi_2$.
\end{proof}

\begin{proposition}
Let  $X_1$ be irreducible in $\mathcal{HC}_{n_1}$ and $X_2$ be irreducible in $\mathcal{HC}_{n_2}$. Let $m_1=\mathrm{ht}(X_1)$ and let $m_2=\mathrm{ht}(X_2)$. Suppose $m_1\neq -\infty$ and $m_2\neq -\infty$. Then 
\begin{align} \label{eqn dimension parabolic functor}
\mathrm{dim}~\Gamma_{n_1+n_2, m_1+m_2}(X_1 \times X_2) =   \mathrm{dim}~(\Gamma_{n_1,m_1}(X_1)\times \Gamma_{n_2,m_2}(X_2)).
\end{align}
where the $\times$ on the LHS of the above equality is real parabolic induction defined in Section \ref{sec-ps}. 
\end{proposition}

\begin{proof}
We only sketch the proof. It is a standard fact that $X_1$ and $X_2$ can be written as a linear combination of standard representations of $\mathrm{GL}_{n_1}(\mathbb C)$ and $\mathrm{GL}_{n_2}(\mathbb C)$ respectively. Thus, one can express the terms $\Gamma_{n_1+n_2, m_1+m_2}(X_1\times X_2)$, $\Gamma_{n_1, m_1}(X_1)$ and $\Gamma_{n_2, m_2}(X_2)$ from such linear combinations and (\ref{eqn dimension as functor}). With the fact that $\chi_{a,b}\times \chi_{a',b'}$ and $\chi_{a',b'}\times \chi_{a,b}$ have the same simple composition factors, it reduces to show the equality for both $X_1$ and $X_2$ to be standard representations. Such case follows from Lemma \ref{lem product hecke algebra} and (\ref{eqn dimension as functor}).
\end{proof}

\subsection{Description on subspaces of the functor} \label{ss description of sub functor}

For fixed $m_1, m_2$, let $S^{m_1, m_2}$ be the set of minimal representatives of $(S_{m_1}\times S_{m_2})\setminus S_{m_1+m_2}$. Let $e_1, \ldots, e_{n_1}, e_{n_1+1}, \ldots, e_{n_1+n_2}$ be the standard basis for $\mathbb C^{n_1+n_2}$. Let $V_1=\mathrm{span}\left\{ e_1, \ldots, e_{n_1} \right\}$ and let $V_2=\mathrm{span}\left\{ e_{n_1+1}, \ldots, e_{n_1+n_2} \right\}$. For $w \in S^{m_1,m_2}$, let $\mathcal X_w$ be the subspace of $\tau \otimes V^{\otimes (m_1+m_2)}$  spanned by the vectors
\[ \mathbf x\otimes  v_1 \otimes v_2 \otimes \ldots \otimes v_{m_1+m_2}
\]
with $v_i \in V_1 $  if $w(i) \in \{1, \ldots, m_1 \}$ and $v_i \in V_2$ if $w(i) \in \{m_1+1,\dots,m_1+m_2\}$, and $\mathbf x \in \tau$.

We have the following simple linear algebra lemma, whose proof is elementary and so omitted.

\begin{lemma} \label{lem injectivity for show paraboic}
 For each representative $ w \in S^{m_1,m_2}$, let $p_w$ be a non-zero element in $\mathcal X_w$. Then those $p_w$'s form a set of linearly independent vectors.
\end{lemma}

\begin{lemma} \label{lem injectivity 2}
We use the notations in the previous lemma. For each representative $w \in S^{m_1,m_2}$, define a linear isomorphism $\iota_w: \mathcal X_1 \rightarrow \mathcal X_w$ given by 
\[ \mathbf x \otimes v_1 \otimes \ldots \otimes v_{m_1+m_2} \mapsto \mathbf x \otimes v_{w^{-1}(1)}\otimes \ldots \otimes v_{w^{-1}(m_1+m_2)} .
\]
The map induces a linear isomorphism from $\mathcal X_1^{K_1\times K_2}$ to $\mathcal X_w^{K_1\times K_2}$.
\end{lemma}

\begin{proof}
This follows from the fact that $K_1 \times K_2$-action commutes with permutations in $S_{m_1+m_2}$.
\end{proof}

\begin{lemma} \label{lem description of subspaces}
Retain the notations in the previous two lemmas. For each $w \in S^{m_1,m_2}$, let $\mathcal W_w^{m_1,m_2}$ be the subspace of $\mathrm{Ind}_P^G(\tau \otimes V^{\otimes (m_1+m_2)})^K$ satisfying
$f(1) \in \mathcal X_w$. Then 
\[\bigoplus_{w \in S^{m_1,m_2}} \mathcal W^{m_1,m_2}_w \subset \mathrm{Ind}_P^G(\tau \otimes V^{\otimes (m_1+m_2)})^K \]
and for each $w$, $\mathrm{dim}~\mathcal W^{m_1,m_2}_w=\mathrm{dim}~\mathcal W^{m_1,m_2}_1$.
\end{lemma}

\begin{proof}
The direct sum follows from Lemma \ref{lem injectivity for show paraboic} and the dimension follows from Lemma \ref{lem injectivity 2}.
\end{proof}

\subsection{Parabolic induction under $\Gamma_{n,m}$: functoriality}

We now study the parabolic induction under $\Gamma_{n,m}$.

\begin{theorem} \label{thm isomophic of parabolic induction}
Let $Y_1$ be in $\mathcal{HC}_{n_1}$ and let $Y_2$ be in $\mathcal{HC}_{n_2}$. Then there exists a natural $\mathbb H_m$-algebra isomorphism:
\[   \Gamma_{n_1+n_2,m}(Y_1 \times Y_2) \cong \bigoplus_{m_1+m_2=m} \Gamma_{n_1,m_1}(Y_1) \times \Gamma_{n_2,m_2}(Y_2) .
\]

\end{theorem}

\begin{proof}




Let $\tau=Y_1\boxtimes Y_2$. For fixed $m_1, m_2$ with $m_1+m_2=m$, let $\mathcal W_w^{m_1, m_2}$ as in Lemma \ref{lem description of subspaces}. Set
\[  \mathcal P_{m_1,m_2}= \bigoplus_{w \in S^{m_1,m_2}} \mathcal W^{m_1,m_2}_w .
\]
Similar to Lemma \ref{lem injectivity for show paraboic}, one sees that
\[   \Gamma_{n_1+n_2,m}(Y_1\times Y_2)\cong (\mathrm{Ind}_P^G(\tau \otimes V^{\otimes m}))^{K} =\bigoplus_{m_1+m_2=m} \mathcal P_{m_1,m_2} .
\]

{\it Claim 1:} Each $\mathcal P_{m_1,m_2}$ is invariant under $\mathbb H_m$-action, and furthermore, 
\[   \mathcal P_{m_1,m_2} \cong \Gamma_{n_1,m_1}(Y_1) \times \Gamma_{n_2,m_2}(Y_2) .
\]

\ \\
Note that the theorem follows from Claim 1. We consider $\mathbb H_{m_1}\otimes \mathbb H_{m_2}$ as a subalgebra of $\mathbb H_{m_1+m_2}$ in Section \ref{ss gaha}. Our proof of Claim 1 relies on another claim: \\

{\it Claim 2:} $\mathcal W_1^{m_1, m_2}$ is invariant under the action of the subalgebra $\mathbb H_{m_1}\otimes \mathbb H_{m_2}$, and is isomorphic to $\Gamma_{n_1,m_1}(Y_1)\boxtimes \Gamma_{n_2,m_2}(Y_2)$. \\

Assume Claim 2 holds in the meanwhile. It is well-known that any element in $\mathbb H_m$ can be written as $\sum_{w \in S^{m_1,m_2}} w^{-1}h_w$ for $h_w \in \mathbb H_{m_1}\otimes \mathbb H_{m_2}$. The element $w^{-1}$ sends $\mathcal W^{m_1,m_2}_1$ to $\mathcal W^{m_1,m_2}_w$ (by a direct computation using Lemma \ref{lem basic formula 1}). This shows that $\mathcal P^{m_1,m_2}$ is invariant under the action of $\mathbb H_m$.

Moreover, by Frobenius reciprocity and the map in Claim 2, we obtain a non-zero map $\psi$ from $\mathbb H_m \otimes_{\mathbb H_{m_1}\otimes \mathbb H_{m_2}}(\Gamma_{n_1,m_1}(Y_1)\boxtimes \Gamma_{n_2, m_2}(Y_2))$ to $\Gamma_{n_1+n_2,m_1+m_2}(Y_1 \times Y_2)$. By Lemma \ref{lem injectivity for show paraboic}, we have that $\psi$ is injective. But, then by comparing the dimensions in Lemmas \ref{lem description of subspaces} and \ref{lem product hecke algebra}, we have that $\psi$ is also surjective. This shows that we have an isomorphism from $\mathcal P^{m_1,m_2}$ to $\Gamma_{n_1, m_1}(Y_1)\times \Gamma_{n_2, m_2}(Y_2)$, proving Claim 1. \\




It remains to prove Claim 2. \\
\noindent
{\it Proof of Claim 2:} We first consider the action of $y_l$ for $1\leq l \leq m_1$ on $\mathcal W_1^{m_1,m_2}$. Let $F \in \mathcal W_1^{m_1,m_2}$. There is a natural isomorphism:
\[  \Phi: \Gamma_{n_1,m_1}(Y_1)\boxtimes \Gamma_{n_2,m_2}(Y_2) \rightarrow \mathcal W_1^{m_1,m_2}
\]
given by 
\[  \mathbf z_1 \otimes \mathbf z_2 \mapsto (1 \mapsto \mathbf z_1 \otimes \mathbf z_2) 
\]
We shall show that $\Phi$ is a $\mathbb H_{m_1}\otimes \mathbb H_{m_2}$-map.

From Definition \ref{def-as}, the action of $y_l$ is given by
\begin{align*}
  \Theta_{\mathbb H_m}(y_l) &=\sum_{k=0}^{l-1} \left( \sum_{1\leq y,z \leq n_1} \Delta_{kl}(E_{yz}) + \sum_{1 \leq y\leq n_1 < z\leq n_1+n_2} \Delta_{kl}(E_{yz})   \right)  \\
    & \quad \quad +\sum_{k=0}^{l-1} \left(\sum_{1\leq z\leq n_1<y \leq n_1+n_2} \Delta_{kl}(E_{yz})+ \sum_{n_1+1\leq y,z\leq n_1+n_2} \Delta_{kl}(E_{yz})  \right)+\frac{n_1+n_2}{2}
\end{align*}

Again $\Theta_{\mathbb H_m}(y_l)F$ is determined by $(\Theta_{\mathbb H_m}(y_l)F)(1)$ and so it suffices to compute that. We shall compute the four parts separately. Now, by Lemma \ref{lem upper corner triangualr comp}, for all $0\leq k<l$,
\[ \left(\sum_{1\leq y\leq n_1<z \leq n_1+n_2}\Delta_{kl}(E_{yz})\cdot F \right)(1)=0 .\]
By Lemmas \ref{lem action of torus lower}, \ref{lem lower upper corner} and \ref{lem lowere lower corner}, for all $0\leq k<l$
\[  \left(\sum_{n_1 < y,z \leq n_1+n_2} \Delta_{kl}(E_{yz})\cdot F\right)(1) =0
\]
By Lemma \ref{lem lower lower action 3} and the fact that $v_k = \sum_{i=1}^{n_1}\langle v_k,e_i\rangle e_i$ (since $v_k \in V_1$), 
\[   \sum_{k=0}^{l-1} \left(\sum_{1 \leq z \leq n_1< y \leq n_1+n_2} \Delta_{kl}(E_{yz}) \cdot F \right)(1)= -n_2F(1)   . \]
As a result, $\Theta_{\mathbb H_m}(y_l)$ is reduced to considering 
\begin{align} \label{eqn action term}
     \left(( \sum_{k=0}^{l-1} \sum_{1\leq y,z \leq n_1} \Delta_{kl}(E_{yz}) +\frac{n_1-n_2}{2})\cdot F\right)(1) .
\end{align}
One may then apply similarly the formula in Lemma \ref{lem basic formula 1}. Putting the normalized factor $\delta^{1/2}$ in the parabolic induction into consideration, one cancels the effect of $-\frac{n_2}{2}$ in the expression of (\ref{eqn action term}), so that it agrees with 
\[   \Theta_{\mathbb H_{m_1}}(y_l) \cdot \Phi^{-1}(F)=\Theta_{\mathbb H_{m_1}}(y_l) \cdot (F(1)).
\]
Here $\Theta_{\mathbb H_{m_1}}(y_l)$ on the LHS is viewed as  the action via the factor in $Y_1\otimes (\mathbb C^{n_1})^{\otimes m_1}$ in $\mathcal W^{m_1,m_2}_1=(Y_1\otimes (\mathbb C^{n_1})^{\otimes m_1})\otimes (Y_2\otimes (\mathbb C^{n_2})^{\otimes m_2})$.

Checking the action for $y_l$ for $m_1+1\leq l\leq m_1+m_2$ is similar and we shall be slightly brief. Again, we divide the computations into four parts as follows:
\begin{itemize}
    \item For the matrices coming from 
$\begin{pmatrix} * & \\ &  \end{pmatrix}$, the action is zero by Lemmas \ref{lem action of torus upper}, \ref{lem upper upper corner} and \ref{lem lower upper actiona }.
 \item  For the matrices coming from $\begin{pmatrix} & * \\ &  \end{pmatrix}$, the action is zero by Lemma \ref{lem upper corner triangualr comp}.
 \item 
For the matrices coming from $\begin{pmatrix} &  \\  * & \end{pmatrix}$, one uses Lemma \ref{lem lower corner action } to get to zero by considering a sum. 
\item  For the matrices $E$ coming from $\begin{pmatrix} &  \\ &  * \end{pmatrix}$, one gets the terms $\Delta_{kl}(E)$ acting by zero for $1\leq k \leq m_1$ by using Lemmas \ref{lem action of torus lower}, and  \ref{lem lower upper corner} and \ref{lem lowere lower corner}. 
\end{itemize}

Then $\Theta_{\mathbb H_m}(y_l)\cdot F$ reduces to: 
\[  \left((\sum_{\substack{k=0, \\ m_1< k< l}}\sum_{n_1+1\leq y,z\leq n_1+n_2} \Delta_{kl}(E_{yz}) + \frac{n}{2}) \cdot F \right)(1) .
\]
With the effect of the normalizing factor $-\frac{n_1}{2}$ from parabolic induction, we have that the term agrees with
\[ \Theta_{\mathbb H_{m_2}}(y_l) \cdot \Phi^{-1}(F) =\Theta_{\mathbb H_{m_2}}(y_l)\cdot (F(1)) 
\]
This also shows the isomorphism in Claim 2. Here $\Theta_{\mathbb H_{m_2}}(y_l)$ on the LHS is viewed as  the action via the factor in $Y_2\otimes (\mathbb C^{n_2})^{\otimes m_2}$ in $\mathcal W^{m_1,m_2}_1=(Y_1\otimes (\mathbb C^{n_1})^{\otimes m_1})\otimes (Y_2\otimes (\mathbb C^{n_2})^{\otimes m_2})$.
\end{proof}

\subsection{General form}

We now have the general form for preserving parabolic inductions.

\begin{corollary} \label{cor preserve parabolic induction}
Let $X_i$ be Harish-Chandra modules of $\mathrm{GL}_{n_i}(\mathbb C)$ for $i=1, \ldots, r$. Then 
\[  \Gamma_{n,m}(X_1\times \ldots \times X_r) = \bigoplus_{m_1,\dots,m_r} \Gamma_{n,m_1}(X_1)\times \ldots \times \Gamma_{n,m_r}(X_r) ,
\]
where $m_1, \ldots, m_r$ run for all non-negative integers such that $m_1+\ldots+m_r=m$.
\end{corollary}

\section{Standard Modules} \label{s std modules}
In this section, we study the image of $\Gamma_{n,m}$ for principal series $X(\lambda_L,\lambda_R)$ of $\mathrm{GL}_n(\mathbb{C})$. In fact, Theorem \ref{thm isomophic of parabolic induction} reduces the study to a character of $\mathrm{GL}_1(\mathbb C)$, whose image under $\Gamma_{1,m}$ will be computed in Section \ref{ss image char gl1}.

\subsection{Reduction to a particular $\Gamma_{n,m}$}

\begin{lemma} \label{lem height zero}
Let $m'=\mathrm{ht}(J(\lambda_L, \lambda_R))$. If $m \neq m'$, then $\Gamma_{n,m}(X(\lambda_L, \lambda_R))=0$ and $\Gamma_{n,m}(J(\lambda_L, \lambda_R))=0$.
\end{lemma}
\begin{proof}
This follows from the computations in Section \ref{ss standard module dimension}.
\end{proof}

With this lemma, we only have to study $\Gamma_{n,m}(X)$ when $m$ is the height of a principal series or an irreducible Harish-Chandra module $X$ in the rest of sections.

\subsection{Image of $\Gamma_{1,m}$ on a character} \label{ss image char gl1}
\begin{lemma} \label{lem basic steinberg case}
Let $\chi_{a,b}$ be a $\mathrm{GL}_1(\mathbb{C})$-character defined in \eqref{eqn character gl1}.
Suppose $m := a-b \geq 0$. Then 
\[  \Gamma_{1,m}(\chi_{a, b})=\mathrm{St}([b+\frac{1}{2}, a-\frac{1}{2}])
\]
\end{lemma}

\begin{proof}
It is clear that $\Gamma_{1,m}(\chi_{a, b})$ is one-dimensional. Now the action follows from a straightforward computation. For example, the element $1$ in $\mathfrak h^* \cong \mathbb C$ acts on $\chi_{a,b}$ according to the formula (\ref{eqn isomorphism 1}), 
\[ \frac{1}{2}[ (b+a) +(\sqrt{-1})^2(a-b)]=b
\]
and hence $y_1$ in $\mathbb H_m$ acts by $b+\frac{1}{2}$. 
 \end{proof}

\subsection{Image of $\Gamma_{n,m}$ on a standard module} 

For $(\lambda_L, \lambda_R) \in \mathfrak h_0^* \times \mathfrak h_0^*$ with $\mathrm{ht}(\lambda_L, \lambda_R)\neq -\infty$, we define a multisegment: write $\lambda_L=(\lambda_{L,1}, \ldots, \lambda_{L,n})$ and $\lambda_R=(\lambda_{R,1},\ldots, \lambda_{R,n})$, 
\[   \mathfrak m(\lambda_L, \lambda_R)=\left\{ [\lambda_{R,1}+\frac{1}{2}, \lambda_{L,1}-\frac{1}{2}], \ldots , [\lambda_{R,n}+\frac{1}{2}, \lambda_{L,n}-\frac{1}{2}] \right\} ,
\]
where we drop all the empty sets (this happens when $\lambda_{L,i}=\lambda_{R,i}$).

\begin{lemma} \label{lem alternate form of std module}
Consider a principal series $X(\lambda_L, \lambda_R)$ of $\mathrm{GL}_n(\mathbb C)$. Write $\lambda_L=(\lambda_{L,1}, \ldots, \lambda_{L,n})$. If 
\[  \mathrm{Re}(\lambda_{L,1}) \geq \ldots \geq \mathrm{Re}(\lambda_{L,n}) ,
\]
then $X(\lambda_L, \lambda_R)$ is a standard module.
\end{lemma}

\begin{proof}
Recall that $\chi_{r_1,s_1} \times \chi_{r_s, s_2}$ is reducible if and only if there exist integers $p, q \in \mathbb Z$ with $p,q >0$ such that $(\chi_{r_1,s_1}\chi_{r_2,s_2}^{-1})(z)=z^p\bar{z}^q$ (see \cite{Go70}). From this, if we have $\mathrm{Re}(r_1+s_1) \geq \mathrm{Re}(r_2+s_2)$ and $\mathrm{Re}(r_1) \leq \mathrm{Re}(s_1)$, we have that $\chi_{r_1,s_1}\times \chi_{r_2, s_2}$ is irreducible and so, by \cite[Proposition 4.1.12]{Vo81},
\begin{equation} \label{eqn commute of characters}
\chi_{r_1,s_1}\times \chi_{r_2,s_2} \cong \chi_{r_2,s_2} \times \chi_{r_1,s_1} .
 \end{equation}
 The proof for the lemma is an elementary exercise by repeatedly using (\ref{eqn commute of characters}) and using the definition of a standard module in Section \ref{sec-ps}, which we leave for the reader.
\end{proof}

\begin{theorem} \label{thm-std}
Let $\lambda_L=(\lambda_{L,1}, \ldots, \lambda_{L,n})$ and $\lambda_R=(\lambda_{R,1}, \ldots, \lambda_{R,n})$. Suppose $m := \mathrm{ht}(X(\lambda_L,\lambda_R)) \neq -\infty$. Let $\Delta_i=[\lambda_{R,i}+\frac{1}{2}, \lambda_{L,i}-\frac{1}{2}]$. Then
$$\Gamma_{n,m}(X(\lambda_L,\lambda_R)) = \mathrm{St}(\Delta_1)\times \ldots \times \mathrm{St}(\Delta_n).$$
In particular, $\Gamma_{n,m}$ maps standard modules $X(\lambda_L, \lambda_R)$ to standard modules $\lambda(\mathfrak m(\lambda_L, \lambda_R))$.
\end{theorem}

\begin{proof}
This follows from
the definition of $X(\lambda_L,\lambda_R)$ in Equation \eqref{eq-psinduce}, Theorem \ref{thm isomophic of parabolic induction} and
Lemma \ref{lem basic steinberg case}. The second statement follows from the definitions of standard modules and Lemma \ref{lem alternate form of std module}.
\end{proof}

\begin{remark}
Note that $\Gamma_{n,m}$ maps tempered representations (see \cite[Page 286]{Knp86} for a definition of tempered representations) to tempered $\mathbb H_m$-modules or zero (see \cite[Definition 1.4]{Ev96} for the definition of tempered $\mathbb H_m$-modules).
In the next section, we will show that $\Gamma_{n,m}$ maps unitary modules to unitary modules.
\end{remark}

\section{Hermitian forms under $\Gamma_{n,m}$}

In this section, we prove that $\Gamma_{n,m}$ preserves Hermitianity and unitarity. 

\subsection{Hermitian form for Harish-Chandra modules}

\begin{definition} \label{def hermitian harish chandra}
Let $X$ be a Harish-Chandra module of $\mathrm{GL}_n(\mathbb C)$. The Hermitian dual $X^h$ of $X$ is the space of $K$-finite vectors of the complex conjugate linear functionals on $X$. There is a non-degenerate sesqui-linear pairing $\langle , \rangle: X \times X^h \rightarrow \mathbb C$ determining $(\mathfrak g, K)$-module structure on $X^h$:
\[ \langle x, (E_1,E_2)\cdot f \rangle =-\langle (\overline{E}_2,\overline{E}_1)\cdot x, f\rangle \]
\[   \langle x, k\cdot f \rangle = \langle k^{-1}\cdot x, f \rangle \]
for any $E_1, E_2 \in \mathfrak g_0$, $k \in K$ and $x \in X$ and $f\in X^h$. We say that $X$ is {\bf Hermitian} if $X \cong X^h$. 
Moreover, we say that $X$ is {\bf unitary} if $X \cong X^h$ and the induced pairing can be chosen to be positive-definite.
\end{definition}

We first see that taking the Hermitian dual preserves the height of an irreducible representation.

\begin{proposition}
$\mathrm{ht}(J(\lambda_L, \lambda_R)) =\mathrm{ht}(J(\lambda_L, \lambda_R)^h) $.
\end{proposition}

\begin{proof}
One has $J(\lambda_L,\lambda_R)^h = J(-\overline{\lambda_R}, -\overline{\lambda_L})$ (c.f. \cite[Theorem 2.4]{Ba89}), so that $J(\lambda_L,\lambda_R)$ is Hermitian if and only if $(\lambda_L,\lambda_R) = w(-\overline{\lambda_R}, -\overline{\lambda_L})$ for some $w \in S_n$ by Theorem \ref{thm-Zh}(2). Now the proposition follows from computing $\mathrm{ht}(J(\lambda_L, \lambda_R))$ and $\mathrm{ht}(J(-\overline{\lambda}_R, -\overline{\lambda}_L))$ directly.
\end{proof}

\subsection{Hermitian form for modules of graded Hecke algebras}


We first need a suitable notion of Hermitian modules for $\mathbb H_m$ (see e.g. \cite{BCT}), translated from $p$-adic groups:

\begin{definition} 
An anti-involution $^*: \mathbb H_m \rightarrow \mathbb H_m$ is a linear map determined by:
\[  s_i^* =s_i,  \quad   y_i^* =-w_0(y_{m+1-i})w_0^{-1}.
\]
\end{definition}

\begin{definition}
Let $\pi$ be an $\mathbb H_m$-module. The Hermitian dual, denoted $\pi^*$, of $\pi$ is the space of complex conjugate linear functionals from $\pi$ to $\mathbb C$ determined by:
\[    (h\cdot f)(x)= f(h^*\cdot x) .
\]
In particular, there exists a non-degenerate sesqui-linear pairing $\langle , \rangle: \pi \times \pi^*\rightarrow \mathbb C$ such that
\[ \langle h\cdot x, f\rangle =\langle x, h^*\cdot f \rangle .\]
We say that $\pi$ is {\bf $*$-Hermitian} if $\pi \cong \pi^*$. We say that $\pi$ is {\bf $*$-unitary} if $\pi$ is $*$-Hermitian and the induced pairing on $\pi$ can be chosen to be positive-definite.
\end{definition}

\subsection{Preserving Hermitian structure}

One may compare the proof with \cite[Theorem 4.2.2]{CT11}.

\begin{theorem} \label{thm preserve hermitian structure}
Let $X$ be in $\mathcal{HC}_n$. Then
\[   \Gamma_{n,m}(X^h) \cong \Gamma_{n,m}(X)^* .\]
\end{theorem}

\begin{proof}
We write $\langle , \rangle$ for the natural pairing between $X$ and $X^h$ as in Definition \ref{def hermitian harish chandra}. Recall that $\langle , \rangle_V$ is the standard inner product on $V$ defined before Lemma \ref{lem lower lower action 3}.

We define a Hermitian form $\langle , \rangle'$ on $X\otimes V^{\otimes m}$:
\begin{align*}
    \langle f\otimes e_{i_1} \otimes \ldots \otimes e_{i_m} , f' \otimes e_{i'_1}\otimes \ldots \otimes e_{i'_m} \rangle' =\langle f, f' \rangle \cdot \langle e_{i_1}, e_{i_1'} \rangle_V \cdot \ldots  \cdot \langle e_{i_m}, e_{i_m'} \rangle_V .
\end{align*}
This descends to a pairing on $\Gamma_{n,m}(X)\times \Gamma_{n,m}(X^h)$ by taking the $K$-invariant.

We have to check that: for $p \in \Gamma_{n,m}(X)$ and $p' \in \Gamma_{n,m}(X^h)$,
\begin{align} \label{eqn dual weyl group}
    \langle \Theta(s_{l,l+1})\cdot p, p' \rangle' = \langle p, \Theta(s_{l,l+1})\cdot p' \rangle' 
\end{align}
\begin{align} \label{eqn dual polynomial}
    \langle \Theta(y_l)\cdot p, p' \rangle' =\langle p, \Theta(y_l^*)\cdot  p' \rangle' .
\end{align}
Checking (\ref{eqn dual weyl group}) is straightforward. 

We now consider (\ref{eqn dual polynomial}). 
We first consider $p$ takes the form:
\[   f \otimes e_{i_1} \otimes \ldots \otimes e_{i_m}\]
and $p'$ takes the form:
\[  f' \otimes e_{i_1'}\otimes \ldots \otimes e_{i_m'}\]
(which are not necessarily in $\Gamma_{n,m}(X)$ or $\Gamma_{n,m}(X^h)$).

Then 
\begin{align*}
\langle \Theta(y_l)\cdot p, p' \rangle' =  & \sum_{1 \leq i,j\leq n} \langle (0,E_{ij})\cdot f, f' \rangle \cdot \langle e_{i_1}, e_{i_1'}\rangle_V \cdot \ldots \cdot \langle (0,E_{ji})\cdot e_{i_l}, e_{i_l'} \rangle_V \cdot \ldots \cdot \langle e_{i_m}, e_{i_m'}\rangle_V  \\
 & \quad + \sum_{1\leq i,j\leq n} \sum_{1 \leq k <l} \langle f, f' \rangle \cdot \ldots \cdot \langle (0,E_{ij})\cdot e_{i_k}, e_{i_k'} \rangle_V \cdot \ldots \cdot \langle (0,E_{ji})\cdot e_{i_l}, e_{i_l'} \rangle_V \cdot \ldots \cdot \langle e_{i_m}, e_{i_m'} \rangle _V \\
=&   \sum_{1 \leq i,j\leq n} \langle (-E_{ji},E_{ij})\cdot f, f' \rangle \cdot \langle e_{i_1}, e_{i_1'}\rangle_V \cdot \ldots  \cdot \langle (0,E_{ji})\cdot e_{i_l}, e_{i_l'} \rangle_V \cdot \cdot \langle e_{i_m}, e_{i_m'}\rangle_V \\
 & \quad + \sum_{1\leq i,j\leq n} \sum_{1 \leq k <l} \langle f, f' \rangle \cdot \ldots \cdot \langle 
(-E_{ji},E_{ij})\cdot e_{i_k}, e_{i_k'} \rangle_V \cdot \ldots \cdot \langle (0,E_{ji})\cdot e_{i_l}, e_{i_l'} \rangle_V \cdot \ldots \cdot \langle e_{i_m}, e_{i_m'} \rangle_V  \\
& \quad  + \sum_{1 \leq i,j \leq n} \langle (E_{ji},0)\cdot f, f' \rangle \cdot \ \ldots \cdot \langle (0, E_{ji})\cdot  e_{i_l}, e_{i_l'} \rangle_V  \cdot \ldots \cdot \langle e_{i_m}, e_{i_m'} \rangle_V +\frac{n}{2} \langle p, p' \rangle' 
\end{align*}
where the second equality follows by using $(E_{ji},0)$ acts by zero on $V$. The first three terms in the last sum will be denoted by (I), (II) and (III) respectively.

We now proceed to compute $\langle p, \Theta(y_l^*)\cdot p'\rangle'$. Note that the action of $w_0$ simply permutes those $e_{i_a'}$ terms. Now, 

\begin{align*}
& \langle p, \Theta(y_l^*)\cdot p' \rangle' \\
=&  - \langle p, \Theta(w_0y_{m+1-l}w_0)\cdot p' \rangle'         \\
=& -\sum_{1\leq i,j \leq n} \langle f, (0,E_{ij})\cdot f'\rangle \cdot \langle e_{i_1}, e_{i_1'} \rangle_V \cdot \ldots \cdot \langle e_{i_l}, (0,E_{ji}).e_{i_l'} \rangle_V \cdot \ldots \cdot \langle e_{i_m}, e_{i_m'} \rangle_V     \\   
  & \quad  - \sum_{1\leq i,j \leq n} \sum_{l<k \leq m}  \langle f, f' \rangle \cdot \ldots \cdot \langle e_{i_l}, (0,E_{ji})\cdot  e_{i_l'} \rangle_V \cdot \ldots \cdot \langle e_{i_k}, (0,E_{ij})\cdot e_{i_k}' \rangle_V \cdot \ldots \cdot \langle e_{i_m}, e_{i_m'} \rangle_V  \\ 
 & \quad -\frac{n}{2}\langle p, p' \rangle \\
=& -\sum_{1\leq i,j \leq n} \langle f, (0,E_{ij})\cdot f'\rangle \cdot \langle e_{i_1}, e_{i_1'} \rangle_V \cdot \ldots \cdot \langle e_{i_l}, (0,E_{ji})\cdot e_{i_l'} \rangle_V \cdot \ldots \cdot \langle e_{i_m}, e_{i_m'} \rangle_V     \\   
  & \quad  - \sum_{1\leq i,j \leq n} \sum_{l<k \leq m}  \langle f, f' \rangle \cdot \ldots \cdot \langle e_{i_l}, (0,E_{ji})\cdot  e_{i_l'} \rangle_V \cdot \ldots \cdot \langle e_{i_k}, (-E_{ji},E_{ij})\cdot e_{i_k}' \rangle_V \cdot \ldots \cdot \langle e_{i_m}, e_{i_m'} \rangle_V  \\ 
 & \quad -\frac{n}{2}\langle p, p' \rangle ,
\end{align*}
where the last equation again follows by $(E_{ij},0)$ acts by zero on $V$. We shall label the first two terms in the last sum by (IV) and (V) (without the subtractions). 

We now consider one more formula:
\begin{align} \label{eqn sum terms equal to n}
\sum_{1 \leq i, j \leq n}  \langle f, f' \rangle \cdot \langle e_{i_1}, e_{i_1'} \rangle_V \cdot  \ldots \cdot \langle (0,E_{ji})\cdot (-E_{ji}, E_{ij})\cdot  e_{i_l}, e_{i_l}' \rangle_V \cdot \ldots \cdot \langle e_{i_l}, e_{i_l}' \rangle_V = & n\cdot \langle p, p' \rangle,
\end{align}
which follows from $(0,E_{ji})\cdot (-E_{ji}, E_{ij})\cdot e_{i_l}=e_{i_l}$ if $j=i_l$ and $=0$ otherwise.

For computing $\langle \Theta(y_l) \cdot p, p' \rangle-\langle p, \Theta(y_l^*)\cdot p' \rangle$, we express the two terms in the above ways. Note that (III) is cancelled with (IV) by using $\langle (0, E_{ji})\cdot f, f' \rangle =-\langle f, (0, E_{ji})\cdot f' \rangle$ and $\langle e_{i_l}, (0,E_{ji})\cdot e_{i_l'}\rangle_V=\langle (0,E_{ij})\cdot e_{i_l}, e_{i_l'}\rangle$. Now the terms (I), (II) and (V) with the formula (\ref{eqn sum terms equal to n}) give that 
\begin{align} \label{eqn a formula for hermitian}
     \langle \Theta(y_l) \cdot p, p' \rangle'-\langle p, \Theta(y_l^*) \cdot p' \rangle' =\sum_{1\leq i, j \leq n} \langle \Delta_l((0,E_{ij}))\cdot ((-E_{ji}, E_{ij})\cdot p), p' \rangle',
\end{align}
We emphasize here that $(-E_{ji}, E_{ij})\cdot p$ means the Lie algebra action of $(-E_{ji}, E_{ij})$ on $p \in X\otimes V^{\otimes m}$.

Now, we consider $p \in \Gamma_{n,m}(X)$ (and $p' \in \Gamma_{n,m}(X^h)$) and the formula (\ref{eqn a formula for hermitian}) extends linearly. However, by using the $K$-invariant, we have that $(-E_{ji}, E_{ij}) \cdot p=0$. In other words, $\langle \Theta(y_l)\cdot p,p'\rangle'=\langle p, \Theta(y_l^*) \cdot p'\rangle'$ as desired.
\end{proof}

\begin{corollary} \label{cor unit}
$\Gamma_{n,m}$ maps Hermitian Harish-Chandra modules to $*$-Hermitian $\mathbb{H}_m$-modules or zero, and maps unitary Harish-Chandra modules to $*$-unitary $\mathbb{H}_m$-modules or zero.
\end{corollary}

\begin{proof}
The Hermitian part follows from definitions and Theorem \ref{thm preserve hermitian structure}. The unitary part follows from the Hermitian form constructed in Theorem \ref{thm preserve hermitian structure} can also be chosen to be positive-definite if we start from a positive-definite Hermitian form on $X$.
\end{proof}

We will discuss non-unitary results in Section \ref{subsec nonunit} after proving the preservation of irreducibility under $\Gamma_{n,m}$.

\begin{remark}
The Lefschetz principle for the unitary dual is also explicated in \cite{Ta09}.
\end{remark}

\begin{remark}
The unitary characters of $\mathrm{GL}_{n'}(\mathbb C)$ are the building blocks of unitary representations of $\mathrm{GL}_n(\mathbb C)$ in the sense that all unitary representations are products of those characters. We shall see from Theorem \ref{thm irreducibility of functor} that the image of a unitary character of $\mathrm{GL}_{n'}(\mathbb C)$ is a (unitary) Speh module of $\mathbb H_m$ (defined in Section \ref{ss dirac series}). It is well-known from \cite{Ta86} that Speh modules play a similar role of building blocks for the unitary dual of $\mathbb H_m$.
\end{remark}

\begin{remark}   \label{rmk-bullet}
As seen in \cite{ALTV20}, even one is interested in studying the unitary dual problem, it is useful to study other Hermitian forms. The anti-involution on $\mathrm{GL}_n(\mathbb C)$ is given by:
\[   g \mapsto \bar{g}^t \]
i.e. the transpose with the complex conjugation. For Harish-Chandra modules, the anti-involution corresponds to the anti-involution $(E,E')\mapsto (\overline{E}^t, \overline{E'}^t)$. On the other hand, there is an anti-involution on $\mathbb H_m$ which is the identity on the generators $s_l$ and $y_l$ in Definition \ref{def gaha type a}. It is more straightforward to check that the Hermitian forms associated to the anti-involutions correspond to each other under the functor $\Gamma_{n,m}$ (c.f. \cite{Su98}).
\end{remark}





\section{Bernstein-Zelevinsky derivatives and Tensor products} \label{s bz derivatives}


We introduce the tool of Bernstein-Zelevinsky derivatives and connect to the tensor product problem in this section. Some connections to the original notion of Bernstein-Zelevinsky derivatives for $p$-adic groups will be discussed in Section \ref{s lefeschetz principle dirac}.

\subsection{Generalized BZ derivatives} \label{ss generalized bz}

For each $\tau \in \mathrm{Irr}(S_i)$, define
\[ \mathbf{BZ}_{\tau}: \mathcal H_n \rightarrow \mathcal H_{n-i} \]
is defined as:
\[   \mathbf{BZ}_{\tau}(\pi)=\mathrm{Hom}_{S_i}(\tau, \pi),
\]
where $\pi$ is regarded as a $\mathbb C[S_i]$-module via the embedding $s_l \mapsto s_l$. The  $\mathbb H_{n-i}$-module structure on $\widetilde{\pi}=\mathbf{BZ}_{\tau}(\pi)$ is via the following actions:
\[    (y_k \cdot_{\widetilde{\pi}} f)(x):=   y_{i+k}\cdot_{\pi} f(x) , \quad    (s_{k} \cdot_{\widetilde{\pi}} f)x := s_{i+k} \cdot_{\pi} f(x) 
\]
for $x \in \tau$. 

When $\tau$ is the sign representation of $S_i$, it is studied from the viewpoint of representations of $p$-adic groups \cite{CS19, Ch22+}. One nice property is the following:

\begin{proposition} \label{prop highest derivatives}
Let $\mathfrak m$ be a multisegment. Write $\mathfrak m=\left\{ \Delta_1, \ldots, \Delta_k \right\}$. For each $\Delta_i=[a_i, b_i]$, write ${}^-\Delta=[a_i+1, b_i]$ and ${}^-\mathfrak m=\left\{ {}^-\Delta_1, \ldots, {}^-\Delta_k \right\}$. Let $\mathrm{triv}=\mathrm{triv}_k$ be the trivial representation of $S_k$. Then 
\[    \mathbf{BZ}_{\mathrm{triv}}(\mathrm{St}(\mathfrak m))\cong \mathrm{St}({}^-\mathfrak m) .
\]
\end{proposition} 

\begin{proof}
This follows from the highest derivative of an irreducible representation due to Zelevinsky \cite[Theorem 8.1]{Ze80}, and the translation to $\mathbb H_m$ in \cite[Theorem 6.9]{CS19}. Note that to translate the version from the sign representation to $\mathrm{triv}$, one also applies the Iwahori-Matsumoto involution (see more discussions in the proof of Theorem \ref{thm bz multiplicity one}).
\end{proof}

\subsection{Arakawa-Suzuki functor and Bernstein-Zelevinsky derivatives} \label{ss arakawa suzuki bz derivatives}

Let $\mathrm{Irr}(S_i)$ be the set of irreducible representations of $S_i$. Recall that, for $\tau \in \mathrm{Irr}(S_i)$, $\mathbb S_{\tau}(V)=\mathrm{Hom}_{S_i}(\tau, V^{\otimes i})$, where $S_i$ acts by sign permutation on $V^{\otimes i}$ i.e. $w.(v_1\otimes \ldots \otimes v_i)=(-1)^{l(w)}v_{w(1)}\otimes \ldots \otimes v_{w(i)}$. We regard $\mathbb S_{\tau}(V)$ as a natural $\mathrm{GL}_n(\mathbb C)$-representation, and the classical Schur-Weyl duality asserts that $\mathbb S_{\tau}(V)$ is either irreducible or zero.

\begin{theorem} \label{thm bz undr arakawa suzuki}
Let $\tau \in \mathrm{Irr}(S_i)$. For $m \geq i$, the following diagram is commutative:
\[
\xymatrix{  \mathcal{HC}_n \ar[r]^{T_{\tau}} \ar[d]^{\Gamma_{n,m}} &  \mathcal{HC}_n \ar[d]^{\Gamma_{n,m-i}} \\  \mathcal H_m \ar[r]^{\mathbf{BZ}_{\tau}}   &  \mathcal H_{m-i} 
}
\]
where $T_{\tau}$ is the functor defined by $T_{\tau}(X) := X \otimes \mathbb S_{\tau}(V)$.
\end{theorem}

\begin{proof}
Since $\Theta(w)$ acts by a sign permutation on the $V^{\otimes i}$, we have the decomposition $X \otimes V^{\otimes m}$:
\[   \bigoplus_{\omega \in \mathrm{Irr}(S_i)} \omega \boxtimes (X \otimes \mathbb S_{\omega}(V) \otimes V^{\otimes (m-i)}) .
\]
Here we may regard as a natural $\mathbb C[S_i] \otimes  U(\mathfrak g_0) \otimes \mathbb C[S_{m-i}]$ representation.

Taking the $K$-invariant on the decomposition and then applying $\mathbf{BZ}_{\tau}$, we have:
\[  \mathbf{BZ}_{\tau} \circ \Gamma_{n,m}(X) = (X \otimes \mathbb S_{\tau}(V) \otimes V^{\otimes (m-i)})^K .
\]
In other words, $$\mathbf{BZ}_{\tau} \circ \Gamma_{n,m}(X) \quad \mbox{ and } \quad \Gamma_{n,m-i}\circ T_{\tau}(X)$$ share the same underlying space. It remains to see that the equality above holds for $\mathbb H_{m-i}$-module action. It is quite straightforward to check that it is true for elements of the form $s_l$. For $y_l$, recall that the action from $\mathbf{BZ}_{\tau} \circ \Gamma_{n,m}(X)$ is given by
\[   \sum_{k=0}^{l-1}  \Omega_{kl}  +\frac{n}{2}
\]
The terms $\sum_{k=0}^{i}\Omega_{kl}$ can be grouped together. Using the Lie algebra action on $X\otimes \mathbb S_{\tau}(V)$, the action of such term agrees respectively with the action $\Omega_{0,l-i}$ from $\Gamma_{n,m-i}\circ T_{\tau}(X)$. The actions of the terms $\Omega_{i+1,l}, \ldots, \Omega_{l-1,l}$ then agree with the ones $\Omega_{1,l-i}, \ldots, \Omega_{l-i-1, l-i}$ from $\Gamma_{n,m-i}\circ T_{\tau}(X)$. Thus, the action of $y_l$ coincides on $\mathbf{BZ}_{\tau}\circ \Gamma_{n,m}(X)$ and $\Gamma_{n,m-i}\circ T_{\tau}(X)$. 
\end{proof}

\section{Irreducibility under $\Gamma_{n,m}$} \label{s preserve irr}

We combine the tools of Theorems \ref{thm-std}, \ref{thm preserve hermitian structure} and \ref{thm bz undr arakawa suzuki} to prove that $\Gamma_{n,m}$ preserves the irreducbility in Theorem \ref{thm irreducibility of functor}. 




\subsection{Non-zero image for the thickened case}

\begin{definition} \label{def thickened}
We write $\lambda_L=(\lambda_{L,1}, \ldots, \lambda_{L,n})$ and $\lambda_R=(\lambda_{R,1}, \ldots, \lambda_{R,n})$. We say that the parameter $(\lambda_L, \lambda_R)$ is {\bf thickened} if $\lambda_{L,i}- \lambda_{R,j} \geq 0$  for any $1 \leq i,j \leq n$.
\end{definition}

\begin{lemma} \label{lem non-zero for thickened}
Let $X=J(\lambda_L, \lambda_R)$, where $(\lambda_L, \lambda_R) \in \mathfrak h_0^* \times \mathfrak h_0^*$ is a thickened parameter. Then $\Gamma_{n,m}(X)$ is non-zero, and has a simple quotient isomorphic to $\mathrm{St}(\mathfrak m(\lambda_L, \lambda_R))$.
\end{lemma}

\begin{proof}
It is clear that $(\lambda_L, w\lambda_R)$ is also thickened for any $w \in S_n$. Without loss of generality, we assume that 
\[   \mathrm{Re}(\lambda_{L,1}) \geq \ldots \geq \mathrm{Re}(\lambda_{L,n}) 
\]
so that $\Gamma_{n,m}(X(\lambda_L, w\lambda_R))$ is a standard module.

By using the thickenedness, $m = \mathrm{ht}(X(\lambda_L, w\lambda_R))$ is not equal to $-\infty$ and is a constant for all $w \in S_n$. Then it follows from Theorem \ref{thm-std} that $\Gamma_{n,m}(X(\lambda_L, w\lambda_R))$ is non-zero.



We pick $w_1, \ldots, w_k$ to be a choice of elements in $S_n$ such that $J(\lambda_L, w_1\lambda_R), \ldots, J(\lambda_L, w_k\lambda_R)$ are all irreducible composition factors in $X(\lambda_L, \lambda_R)$ without redundancy. By Theorems \ref{thm-Zh}(b), \ref{thm classify irred hecke} and \ref{thm-std}, $\left\{ \Gamma_{n,m}(X(\lambda_L, w_i\lambda_R)) \right\}_i$ is linearly independent in the Grothendieck group of $\mathbb H_m$-modules. Moreover, the multisegments $\left\{ \mathfrak m(\lambda_L, w_i\lambda_R) \right\}_i$ are closed under intersection-union operations on multisegments. In other words, by \cite[Theorem 7.1]{Ze80}, we have that:
\[  \mathrm{span}_{\mathbb C}\left\{ \mathrm{St}(\mathfrak m(\lambda_L, w_i\lambda_R)) \right\}_i =\mathrm{span}_{\mathbb C}\left\{ \lambda(\mathfrak m(\lambda_L, w_i\lambda_R)) \right\}_i ,
\]
where we regard the representations in the Grothendicek group of $\mathcal H_m$ over $\mathbb C$.

Combining above, we must then have:
$ \Gamma_{n,m}(J(\lambda_L, w_i\lambda_R))$ is non-zero for all $i$. By the exactness of $\Gamma_{n,m}$ and Lemma \ref{lem alternate form of std module}, we then have a non-zero surjection from $\Gamma_{n,m}(X(\lambda_L, w_i\lambda_R))$ to $\Gamma_{n,m}(J(\lambda_L, w_i\lambda_R))$. This gives the proposition by Theorem \ref{thm-std}.
\end{proof}

\subsection{Non-zero image for the general case} \label{ss defining thickening character}

We shall use the tool of Bernstein-Zelevinsky derivatives to deduce the general case.

Let $\chi: \mathrm{GL}_n(\mathbb C) \rightarrow \mathbb C^{\times}$ be the character $\chi(g)=\overline{\mathrm{det}(g)}^{-1}$. Let $\mathrm{triv}_k$ be the trivial representation of $S_k$.


\begin{proposition} \label{prop non-zero general}
Suppose $m := \mathrm{ht}(J(\lambda_L, \lambda_R))\neq -\infty$. Then  $\Gamma_{n,m}(J(\lambda_L, \lambda_R))$ is non-zero. Moreover, $\Gamma_{n,m}(J(\lambda_L, \lambda_R))$ has a unique simple quotient isomorphic to $\mathrm{St}(\mathfrak m(\lambda_L, \lambda_R))$.
\end{proposition}

\begin{proof}
Let $\lambda_R'=\lambda_R-(k,\ldots, k)$ for sufficiently large $k$ such that $(\lambda_L, w\lambda_R')$ is thickened for any $w \in S_n$. By Lemma \ref{lem non-zero for thickened}, we have that $\Gamma_{n,m+kn}(J(\lambda_L, \lambda_R')) \neq 0$ and has the simple composition factor $\mathrm{St}(\mathfrak m(\lambda_L, \lambda_R))$.  Then, one applies the derivative $\mathbf{BZ}_{\mathrm{triv_n}}$ $k$-times and the resulting module 
\[   \mathbf{BZ}_{\mathrm{triv}_n}\circ \dots \circ \mathbf{BZ}_{\mathrm{triv}_n}\circ \Gamma_{n,m+kn}(J(\lambda_L, \lambda_R'))
\]
is still non-zero. Note that $\mathbb{S}_{\mathrm{triv}_n}(V) = \wedge^n V=\chi^{-1}$ and $J(\lambda_L, \lambda_R')\otimes \chi^{\otimes (-k)}=J(\lambda_L, \lambda_R)$. 
Now the statement then follows by using $\mathbf{BZ}_{\mathrm{triv}_n}\circ \Gamma_{n,m'+n}(X)=\Gamma_{n,m'}(\chi^{-1} \otimes X)$ several times from Theorem \ref{thm bz undr arakawa suzuki} and Proposition \ref{prop highest derivatives}.
\end{proof}

\subsection{Irreducibility under $\Gamma_{n,m}$}

\begin{theorem} \label{thm irreducibility of functor}
Let $X=J(\lambda_L, \lambda_R)$ be an irreducible Harish-Chandra module of $\mathrm{GL}_n(\mathbb C)$ such that $\mathrm{ht}(X) \neq -\infty$. Let $m=\mathrm{ht}(X)$. Then $\Gamma_{n,m}(X)$ is irreducible, and moreover, $\Gamma_{n,m}(X) \cong \mathrm{St}(\mathfrak m(\lambda_L, \lambda_R))$. 
\end{theorem}

\begin{proof}
We assume that $\mathrm{Re}(\lambda_L+\lambda_R)$ is dominant. By Theorem \ref{thm preserve hermitian structure}, we have:
\begin{align}  \label{eqn hermitian dual in irr}
\Gamma_{n,m}(J(-\overline{\lambda_R}, -\overline{\lambda_L})) \cong \Gamma_{n,m}(J(\lambda_L, \lambda_R)^h) \cong \Gamma_{n,m}(J(\lambda_L, \lambda_R))^*
\end{align}

By Proposition \ref{prop non-zero general}, $\Gamma_{n,m}(J(\lambda_L, \lambda_R))$ has a unique simple quotient $\mathrm{St}(\mathfrak m(\lambda_L, \lambda_R))$. Then $\Gamma_{n,m}(J(\lambda_L, \lambda_R))^*$ has a unique simple submodule $\mathrm{St}(\mathfrak m(\lambda_L, \lambda_R))^* \cong  \mathrm{St}(\mathfrak m(-\overline{\lambda_R}, -\overline{\lambda_L}))$. By (\ref{eqn hermitian dual in irr}), $\Gamma_{n,m}(J(-\overline{\lambda_R}, -\overline{\lambda_L}))$ also has such unique simple submodule.

On the other hand, by Theorem \ref{thm-std} and the exactness of $\Gamma_{n,m}$, $\Gamma_{n,m}(J(-\overline{\lambda_R},-\overline{\lambda_L}))$ has a unique simple quotient isomorphic to $\mathrm{St}(\mathfrak m(-\overline{\lambda_R},-\overline{\lambda_L}))$, and no other composition factor isomorphic to that. Thus, combining with the previous paragraph, we then have that 
\[  \Gamma_{n,m}(J(-\overline{\lambda_R}, -\overline{\lambda_L})) \cong \mathrm{St}(\mathfrak m(-\overline{\lambda_R},-\overline{\lambda_L})) .
\]
In particular, it is irreducible. Since $(\lambda_L, \lambda_R)$ is arbitrary, we are then done.
\end{proof}

As a result, we also have the following consequence:

\begin{corollary} \label{cor std implying irreducible}
If $\Gamma_{n,m}(X(\lambda_L, \lambda_R))\neq 0$, then $\Gamma_{n,m}(J(\lambda_L, \lambda_R))\neq 0$. 
\end{corollary}

\begin{remark} \label{rmk irred}
\begin{enumerate}
\item The approach for proving irreducibility in \cite{AS98, CT12} is to compare the geometry to obtain coincidence of character formulas. One may also go the other way around. For example, with the results of Theorems \ref{thm-std} and \ref{thm irreducibility of functor} in our setting, and the exactness of $\Gamma_{n,m}$, we obtain the coincidence of character formulas for $J(\lambda_L, \lambda_R)$ and $\Gamma_{n,m}(J(\lambda_L, \lambda_R))$.
\item The analogous statement of Corollary \ref{cor std implying irreducible} does not hold for the original Arakawa-Suzuki functor \cite{AS98} in general.
\item In general, for two irreducible Harish-Chandra modules $X_1$ and $X_2$ of $\mathrm{GL}_n(\mathbb C)$, it is possible that $\Gamma_{n,m}(X_1)\cong \Gamma_{n,m}(X_2)$, but $X_1 \not\cong X_2$. For example, when $n=1$, one may take $X_1=\chi_{r,r}$ and $X_2=\chi_{s,s}$ for $r\neq s$. However, this happens only for those 'degenerate' cases. For example, if both $X_1$ and $X_2$ are thickened with that all those inequalities for $\lambda_{L,i}-\lambda_{R,j}$ in Definition \ref{def thickened} are all strict, then we have $\Gamma_{n,m}(X_1)\cong \Gamma_{n,m}(X_2)$ if and only if $X_1 \cong X_2$. 
\end{enumerate}
\end{remark}

\subsection{Digression: Matching $U(n)$-types and $S_m$-types}


Indeed, the ideas in Section \ref{ss arakawa suzuki bz derivatives} above could also be used to match the $U(n)$-structure of $X$ and the $S_m$-structure of $\Gamma_{n,m}(X)$. Recall that $K=U(n)$. For a $K$-representation $Y$, we denote by $Y^*$ the contragredient $K$-representation of $Y$.

\begin{theorem} \label{thm equality of k type and sm type}
Let $X$ be in $\mathcal{HC}_n$. For any $\tau \in \mathrm{Irr}(S_i)$, we have that:
\[  \mathrm{Hom}_{U(n)}(\mathbb S_{\tau}(V)^*, X|_{U(n)}) \cong \mathrm{Hom}_{S_m}(\tau, \Gamma_{n,m}(X)|_{S_m}) .
\]
\end{theorem}

\begin{proof}
We have that 
\[    \mathrm{Hom}_{S_m}(\tau, \Gamma_{n,m}(X)) \cong \mathbb (X\otimes \mathbb S_{\tau}(V))^K \cong \mathrm{Hom}_K(X, \mathbb S_{\tau}(V)^*) ,
\]
where the first isomorphism is similar to discussions in the proof of Theorem \ref{thm bz undr arakawa suzuki} and the second one follows from the adjointness between Hom and tensor product.
\end{proof}

\begin{remark}
The highest weight of $V$, as a $K$-representation, is $(0, \ldots, 0, -1)$. The irreducible representation of $S_m$ can be parameterized by partitions of $m$ so that the trivial representation corresponds to $(m)$ and the sign representation corresponds to $(1,\ldots, 1)$. 

For a thickened parameter $(\lambda_L, \lambda_R)$ with integral $\lambda_L-\lambda_R$, write $\lambda_L=(\lambda_{L,1}, \ldots, \lambda_{L,n})$ and $\lambda_R=(\lambda_{R,1}, \ldots, \lambda_{R,n})$. This determines a partition $\alpha=(\lambda_{L,1}-\lambda_{R,1}, \ldots, \lambda_{L,n}-\lambda_{R,n})$, whose transpose determines an irreducible representation, denoted $\tau_{\alpha}$, of $S_m$.

The classical Schur-Weyl duality gives that $\mathbb S_{\tau_{\alpha}}(V)^*$, as a $K$-representation, has the highest weight $\lambda_L-\lambda_R$. On the other hand, Theorems \ref{thm irreducibility of functor} and \ref{thm equality of k type and sm type} give
\[   \mathrm{Hom}_{U(n)}(\mathbb S_{\tau_{\alpha}}(V)^*, J(\lambda_L, \lambda_R)) \cong \mathrm{Hom}_{S_m}(\tau_{\alpha}, \mathrm{St}(\mathfrak m(\lambda_L, \lambda_R))) .
\]
This may be seen as a theoretical explanation on Remark \ref{rmk construction of irreducible from S_m types}. 
\end{remark}



\subsection{Non-unitarity under the functor}  \label{subsec nonunit}
For $d \in \mathbb{Z}$, define the unitary character 
$\mathrm{Sp}_{n,d}: \mathrm{GL}_n(\mathbb C)\rightarrow S^1$ by:
\[ \mathrm{Sp}_{n,d}(g) := \left(\frac{\det(g)}{|\det(g)|}\right)^{d} =\chi_{\frac{d}{2},-\frac{d}{2}}(\mathrm{det}(g)),
\]
where $\chi_{\frac{d}{2}, -\frac{d}{2}}$ is given in Definition \ref{def ps hc}. In terms of Theorem \ref{thm-Zh}, 
\[ \mathrm{Sp}_{n,d} = J((\frac{(n-1)+d}{2},\frac{(n-3)+d}{2},\dots,\frac{-(n-1)+d}{2}),
(\frac{(n-1)-d}{2},\frac{(n-3)-d}{2},\dots,\frac{-(n-1)-d}{2})).
\]

\begin{theorem} \label{thm unittounit}
Let $X$ be a Hermitian Harish-Chandra module for $\mathrm{GL}_n(\mathbb C)$. Then, for any integer $k$, $X\otimes \mathrm{Sp}_{n,k}$ is unitary if and only if $X$ is unitary. Moreover, if $X$ is non-unitary, then for sufficiently large $k$, $\Gamma_{n,m}(X\otimes \mathrm{Sp}_{n,k})$ is non-zero and is a non-$*$-unitary $\mathbb H_m$-module, where $m=\mathrm{ht}(X\otimes \mathrm{Sp}_{n,k})$.
\end{theorem}

\begin{proof}
 Since $X$ is Hermitian, there exists a non-degenerate pairing $\langle , \rangle$ as in Definition \ref{def hermitian harish chandra}. We shall normalize such that restriction to the lowest $K$-type, denoted by $\tau_0$, is positive definite.


 For the first assertion, one simply use the Hermitian forms on $X$ and
 $\mathrm{Sp}_{n,k}$ to define the Hermitian form on $X \otimes \mathrm{Sp}_{n,k}$. We shall denote such Hermitian form on  $X\otimes \mathrm{Sp}_{n,k}$ by $\langle , \rangle_k$. It is straightforward to check $\langle , \rangle_k$ satisfies the conditions in Definition \ref{def hermitian harish chandra}. Since $X\otimes \mathrm{Sp}_{n,k}$ is irreducible, the non-degenerate Hermitian form on $X\otimes \mathrm{Sp}_{n,k}$ is unique (up to a scalar). Moreover, $X\otimes \mathrm{Sp}_{n,k}$ is non-unitary if and only if $X$ is non-unitary, since the signatures of the Hermitian forms on $\langle, \rangle$ and $\langle, \rangle_k$ coincide. 

We now prove the second assertion. Suppose $X$ is non-unitary, then there exists a $K$-type $\tau$ such that the pairing $\langle, \rangle$ restricted to the $\tau$-isotypic component is not positive definite. 

   We fix the Cartan subalgebra $\mathfrak h$ of $\mathfrak k$ in the notation (\ref{identification}) in Section \ref{ss notation cplx gp}, and so $\mathfrak h \cong \mathbb C^n$ and the weight space of $\mathfrak h$ is naturally identified with $\mathbb C^n$. Let the highest weight of $\tau_0$ be $(x_1, \ldots, x_n)$ and the highest weight of $\tau$ is $(y_1, \ldots, y_n)$.

   Now, the highest weight of $\tau_0 \otimes \mathrm{Sp}_{n,k}$ (resp. $\tau\otimes \mathrm{Sp}_{n,k}$) is $(x_1+k, \ldots, x_n+k)$ (resp. $(y_1+k, \ldots, y_n+k)$). Then for sufficiently large $k$, all $x_1+k, \ldots, x_n+k, y_1+k, \ldots, y_n+k$ are positive so that 
   \[   ((\tau_0\otimes \mathrm{Sp}_{n,k})\otimes V^{\otimes m})^K\neq 0
   \]
   and
   \[  ((\tau \otimes \mathrm{Sp}_{n,k})\otimes V^{\otimes m})^K \neq 0 .
   \]
   We shall fix such $k$, and let $\omega_0$ (resp. $\omega$) be the $\tau_0\otimes \mathrm{Sp}_{n,k}$-isotypic component (resp. $\tau \otimes \mathrm{Sp}_{n,k}$) in $X \otimes \mathrm{Sp}_{n,k}$.

By our previous construction, $\langle , \rangle_k$ is positive definite on $\omega_0 \times \omega_0$, but $\langle , \rangle_k$ is negative-definite on some $K$-subspace $\omega'$ of $\omega$. Now it follows from the construction in Theorem \ref{thm preserve hermitian structure} that the induced Hermitian form on $\Gamma_{n,m}(X\otimes \mathrm{Sp}_{n,k})$ is positive-definite on $(\omega_0 \otimes \mathrm{Sp}_{n,k})^K\neq 0$ and is negative definite on $(\omega' \otimes \mathrm{Sp}_{n,k})^K\neq 0$. Hence, $\Gamma_{n,m}(X\otimes \mathrm{Sp}_{n,k})$ is indefinite. By Theorem \ref{thm irreducibility of functor}, $\Gamma_{n,m}(X\otimes \mathrm{Sp}_{n,k})$ is irreducible and so a non-degenerate Hermitian form on an irreducible $\mathbb H_n$-module must be the induced one (up to a scalar). This shows that $\Gamma_{n,m}(X\otimes \mathrm{Sp}_{n,k})$ is non-$*$-unitary, as desired.
\end{proof}

An immediate consequence of the above theorem is the following result on detecting unitarity and non-unitarity of any irreducible, Hermitian $(\mathfrak{g},K)$-modules by its image under $\Gamma_{n,m}$:
\begin{corollary}
    Let $X \in \mathcal{HC}_n$ be an irreducible, Hermitian module. Then $X$ is unitary if and only if for all unitary characters $\mathrm{Sp}_{n,k}$ $(k \in \mathbb{N})$,
    $\Gamma_{n,m}(X \otimes \mathrm{Sp}_{n,k}) \in \mathcal{H}_m$ is $*$-unitary for all $m \in \mathbb{N}$. 
\end{corollary}
\begin{proof}
    If $X$ is unitary, then $X \otimes \mathrm{Sp}_{n,k}$ is unitary for all $k \in \mathbb{N}$. Therefore, $\Gamma_{n,m}(X \otimes \mathrm{Sp}_{n,k})$ is unitary or zero by Corollary \ref{cor unit}. On the other hand, if $X$ is not unitary, then there exists $k, m \in \mathbb{N}$ such that $\Gamma_{n,m}(X \otimes \mathrm{Sp}_{n,k})$ is not unitary by Theorem \ref{thm unittounit}. Consequently, the result follows. 
\end{proof}

\section{Lefschetz principle for Dirac cohomology}  \label{s lefeschetz principle dirac}
In the 1990s, Vogan introduced the notion of \emph{Dirac cohomology} for $(\mathfrak{g},K)$-modules. The relationship between the infinitesimal character of a $(\mathfrak{g},K)$-module and its Dirac cohomology is studied
in \cite{HP} in full detail. Later on, an analogous study of Dirac cohomology for graded Hecke algebra is carried out in \cite{BCT}. In this section, we will use $\Gamma_{n,m}$ to relate various results of Dirac cohomologies on representations of $\mathrm{GL}_n(\mathbb C)$ and $\mathbb H_m$.

\subsection{Finite-dimensional representations versus ladder representations} 

An irreducible $\mathbb H_m$-module $\mathrm{St}(\mathfrak m)$ is said to be a {\bf ladder representation} (in the sense in \cite{LM16}) if the multisegment 
\[  \mathfrak m=\left\{ [a_1, b_1], \ldots, [a_n,b_n] \right\}
\]
satisfies that $a_1>a_2>\ldots >a_n$ and $b_1>b_2>\ldots >b_n$.

One can relate finite dimensional representations of $\mathrm{GL}_n(\mathbb{C})$ (see Appendix \ref{appendix} for details) and ladder representations using $\Gamma_{n,m}$. Namely, by Theorem \ref{thm irreducibility of functor} and the first paragraph of Appendix \ref{appendix}, all ladder representations can be realized as the image of some finite-dimensional representations of $\mathrm{GL}_n(\mathbb{C})$ under $\Gamma_{n,m}$. 

By Theorem \ref{thm-finitedim} in the appendix, all finite-dimensional representations of complex groups with nonzero Dirac cohomology must be of the form
$X = J(\lambda,-w_0\lambda)$ for regular integral $\lambda$. So we are interested in studying Dirac cohomology of the ladder representations $\Gamma_{n,m}(X) = \mathrm{St}(\mathfrak{m}(\lambda,-w_0\lambda))$ with $\mathrm{ht}(X) = m > 0$.
In order to invoke the classification of all ladder representations with nonzero Dirac cohomology given in \cite[Section 7]{Ch18}, we define the {\it content} of a multisegment $\mathfrak{m} = \{[a_1,b_1], \dots, [a_n,b_n]\}$ to be the multiset 
\[\mathrm{con}(\mathfrak{m}) := \{a_1, a_1+1, \dots, b_1, \cdots, a_n, a_n+1, \dots, b_n\}\]
(if $a_i = b_i$, we only count it once in the content). And we say $\mathfrak{m}$ is {\bf twisted-elliptic} if there exists a multisegment of the form $\mathfrak{m}^{temp} =\{[-c_1,c_1],\dots,[-c_m,c_m]\}$ such that
$\mathrm{con}(\mathfrak{m}) = \mathrm{con}(\mathfrak{m}^{temp})$ (c.f. \cite[Corollary 7.3]{Ch18}).

For example, $\mathfrak{m} = \{[3,4],[0,1],[-1,0],[-4,-3]\}$
has content $\{-4,-3,-1,0,0,1,3,4\}$,
and is not twisted-elliptic.
However, $\mathfrak{m}'' = \{[0,7],[-3,4],[-4,3],[-7,0]\}$ is twisted-elliptic, with
\[\mathrm{con}(\mathfrak{m}'') = \mathrm{con}(\{[-7,7],[-4,4],[-3,3],[0,0]\}).\]

\begin{lemma}\label{lem ladder}
    Let $X = J(\lambda,-w_0\lambda)$ be such that $\lambda$ is regular and integral. Then the multisegment $\mathfrak{m}(\lambda,-w_0\lambda)$ is twisted-elliptic if and only if the following holds for $\lambda = (\lambda_1 > \lambda_2 > \dots > \lambda_n)$:
    \begin{center}
        If $0 > \lambda_{i+1} > \dots > \lambda_n$, then there exists $a_1 < \dots < a_{n-i}$ such that $\lambda_{a_j} = -\lambda_{n+1-j}$.
    \end{center}
\end{lemma}
\begin{proof}
    This can be checked directly by the definition of twisted-elliptic multisegments. For instance, for the {\it if} part of the lemma, let $\lambda^-$ be obtained by removing all $\lambda_{a_j}$ and $\lambda_{n+1-j}$'s from $\lambda$. Then it can be checked that
    \[\mathrm{con}(\mathfrak{m}(\lambda,-w_0\lambda)) = \mathrm{con}(\mathfrak{m}(\lambda^-,-w_0\lambda^-)),\]
    so that one is reduced to the case when there are no negative coordinates of $\lambda$. In such a case, one can easily prove that the lemma holds. The {\it only if} part can be proved similarly.
\end{proof}

\begin{remark}
    For {\bf all} finite-dimensional representations $X = J(\lambda,-w_0\lambda)$ with nonzero Dirac cohomology, one can always `thicken' $X$ by, for instance, tensoring with a unitary character (see Lemma \ref{lem-speh} below), so that the thickened module $X'' = J(\lambda'',-w_0\lambda'')$ with the $i$-th coordinate $\lambda_i'' \geq 0$ for all $i$, satisfying the hypothesis of the above Lemma.
\end{remark}

\begin{corollary} \label{cor ladder}
Let $X = J(\lambda,-w_0\lambda)$ be a finite-dimensional representation of $\mathrm{GL}_n(\mathbb{C})$ such that $\mathrm{ht}(X) = m > 0$. Then  $\Gamma_{n,m}(X)$ has nonzero Dirac cohomology if and only if the multisegment $\mathfrak{m}(\lambda,-w_0\lambda)$ is twisted-elliptic.

Furthermore, if $\lambda$ is such that $2(\lambda-\rho)$ has nonnegative coordinates, then the $W$-type of $\Gamma_{n,m}(X)$ contributing to its Dirac cohomology is $\tau_{2(\lambda - \rho)}$. 
\end{corollary}
\begin{proof}
The first statement is immediate from  \cite[Theorem 7.8]{Ch18}. For the second statement, note that the proof of Theorem \ref{thm-finitedim} implies that $F_{2(\lambda-\rho)}$ (appearing with multiplicity one) contributes to the Dirac cohomology of $X$. 
By the Schur-Weyl duality (Theorem \ref{thm equality of k type and sm type}),  $\tau_{2(\lambda - \rho)}$ appears in $\Gamma_{n,m}(X)$ with multiplicity one. Then the result follows from checking that the multiplicity
\[\dim \mathrm{Hom}_{\widetilde{W}}(\tau_{\lambda(\mathfrak{m}^{temp})}, \tau_{2(\lambda - \rho)} \otimes S)\]
is nonzero (c.f. \cite[Lemma 7.20]{Ch18}).
\end{proof}

\begin{example}
    Let $X = J((7,3,-3),(3,-3,-7))$
    be a finite dimensional $\mathrm{GL}_3(\mathbb{C})$-module. Then $\mathfrak{m}((7,3,-3),(3,-3,-7))$ is twisted-elliptic by Lemma \ref{lem ladder}, with 
    \[\mathrm{con}(\mathfrak{m}((7,3,-3),(3,-3,-7))) = \mathrm{con}(\mathfrak{m}((7),(-7))) = \left\{-\frac{13}{2}, -\frac{11}{2}, \dots, \frac{11}{2}, \frac{13}{2}\right\}. \]
    
    By the first paragraph of Corollary \ref{cor ladder}, $\Gamma_{3,14}(X) = \mathrm{St}(\mathfrak{m}((7,3,-3),(3,-3,-7)))$ has nonzero Dirac cohomology - indeed, the character formula of $X$ is given by (see \eqref{eq char form} in the Appendix):
    \begin{align*} 
    \sum_{s \in S_3} \det(s) X((7,3,-3),s(3,-3,-7)) = &\ X((7,3,-3),(3,-3,-7)) - X((7,3,-3),(-3,3,7)) \\
    - &X((7,3,-3),(3,-7,-3)) + X((7,3,-3),(-3,-7,3))\\ 
    + &X((7,3,-3),(-7,3,-3)) - X((7,3,-3),(-7,-3,3)).\end{align*}
Consequently, one can obtain the character formula of $\Gamma_{3,14}(X)$ by applying Theorem \ref{thm-std} on the above terms. 
Indeed, the last summand of the above formula maps to zero by $\Gamma_{3,14}$, while the second last summand yields a tempered module $\mathrm{St}([-\frac{13}{2},\frac{13}{2}])$. Note that it is the only tempered summand in the character formula of $\Gamma_{3,14}(X)$. 
As a result,
\cite[Theorem 6.4]{Ch18} implies that the Dirac index (and hence the Dirac cohomology) of $\Gamma_{3,14}(X)$ is nonzero. 

Since $X$ does not satisfy the second hypothesis of Corollary \ref{cor ladder}, one has to invoke the algorithm in \cite[Section 7.6]{Ch18} to conclude that 
the $W$-type in $\Gamma_{3,14}(X)$ contributing to Dirac cohomology is
$\tau_{(12,1,1)}$, occurring with multiplicity one. On the other hand, the decomposition of $X|_K \cong F_{(6,3,-2)} \otimes F_{(6,3,-2)}$ into $K$-types contains $F_{(12,1,1)}$
with multiplicity one. This matches with the Schur-Weyl duality given in Theorem \ref{thm equality of k type and sm type}.

Finally, if one considers the `thickened' module
$X'' := J((10,6,0),(0,-6,-10)) = X \otimes (\frac{\det}{|\det|})^6$, then $X''$ satisfies both hypotheses of Corollary \ref{cor ladder}, and the $W$-type contributing to the Dirac cohomology of $\Gamma_{3,32}(X'')$ is $\tau_{2(\lambda-\rho)} = \tau_{(18,12,2)}$.
\end{example}

\subsection{Dirac series} \label{ss dirac series}
An interesting part of the
unitary dual is called the \emph{Dirac series}, that is, unitary representations with nonzero Dirac cohomology.
As we see in the previous section, the Arakawa-Suzuki functor $\Gamma_{n,m}$ maps irreducible, unitary representations of $\mathrm{GL}_n(\mathbb C)$ to irreducible, unitary representations $\mathbb{H}_m$. As a special case, one may apply Theorem \ref{thm irreducibility of functor} and get:
\begin{lemma} \label{lem-speh}
Let $\mathrm{Sp}_{n,d}$ be the unitary character in $\mathrm{GL}_n(\mathbb C)$ defined in Section \ref{subsec nonunit}. 
Suppose $d > 0$, so that
$m:= \mathrm{ht}(\mathrm{Sp}_{n,d})= nd > 0$. Then
$\Gamma_{n,m}(\mathrm{Sp}_{n,d}) = a(n,d)$, where 
$$a(n,d) := \mathrm{St}\left(\{[\frac{n-d}{2},\frac{n+d}{2}-1], [\frac{n-d}{2}-1,\frac{n+d}{2}-2], \dots, [-\frac{n+d}{2}+1,-\frac{n-d}{2}]\}\right)$$ 
is the {\bf Speh module}  (here we use the notations in \cite[Section 3.2]{BC14}) .
\end{lemma}

In Type $A$, the classification of Dirac series for $\mathrm{GL}_n(\mathbb C)$
and $\mathbb{H}_m$ are given in \cite{BP}, \cite{DW} and \cite{BC14} respectively. In particular, one has:

\begin{corollary}
Let $X$ be in the Dirac series of $\mathrm{GL}_n(\mathbb C)$
such that $\mathrm{ht}(X) = m >0$. Then $\Gamma_{n,m}(X) \neq 0$ is also in the Dirac series of $\mathbb{H}_m$.
\end{corollary}
\begin{proof}
By the description of Dirac series given in the beginning of \cite[Section 2]{DW}, $\pi$ is in the Dirac series if and only if
$$X \cong (\mathrm{Sp}_{n_1,d_1} \times \dots \times \mathrm{Sp}_{n_l,d_l}) \times (\mathrm{Sp}_{n_1',d_1'} \times \dots \times \mathrm{Sp}_{n_t',d_t'}),$$
where $n_i + d_i - 1 \equiv 0 (\mathrm{mod}\ 2)$ for all $1 \leq i \leq l$, 
$n_j' + d_j' - 1 \equiv 1 (\mathrm{mod}\ 2)$ for all $1\leq j \leq t$, and
\begin{align*}
n_1+d_1-1 \geq -n_1+d_1+1 > n_2 + d_2 - 1 \geq -n_2 + d_2 -1 > \dots  > n_l+d_l-1 \geq -n_1+d_1+1; \\
n_1'+d_1'-1 \geq -n_1'+d_1'+1 > n_2' + d_2' - 1 \geq -n_2' + d_2' -1 > \dots  > n_t'+d_t'-1 \geq -n_t'+d_t'+1.
\end{align*}
On the other hand, Lemma \ref{lem-speh} implies that
$$\Gamma_{n,m}(X) = (a(n_1,d_1) \times \dots \times a(n_l,d_l)) \times (a(n_1',d_1') \times \dots \times a(n_t',d_t')).$$
Since $n+d-1 \geq |n-d|+1 \geq -n+d+1$ whenever $n,d > 0$, it is easy to see that the inequalities above satisfy Equation (3.6.2) of \cite{BC14}, and hence $\Gamma_{n,m}(X)$ has nonzero Dirac cohomology.
\end{proof}

It would be beneficial to give a theoretical account of the results in this section without invoking \cite{BC14}, \cite{Ch18} or \cite{DW}. It is of interest to see how the Dirac operator `transforms' upon applying the Arakawa-Suzuki functor. As seen in \cite{BCT}, the generators $y_l$ in Definition \ref{def gaha type a} play the role of differential operators for real groups. Hence, a more general question is to see how the work of \cite{Hu93} and \cite{BKZ09}-- which roughy speaking uses differential operators to  study reducibility for generalized Verma modules -- can be transferred to the $\mathbb H_{m}$-modules.

\section{Applications from Bernstein-Zelevinsky derivatives} \label{s bz derivatives applications}

\subsection{Tensor products on $\mathrm{GL}_n(\mathbb C)$} \label{ss bz derivatives}

In this section, we study $\mathbf{BZ}_{\tau}$ for $\tau=\mathrm{triv}$, the trivial representation of $S_i$ and $\mathrm{sgn}$, the sign representation. We first recall the following multiplicity-free result:

\begin{theorem} \label{thm bz multiplicity one} \cite{CS19, Ch21, Ch22+}
For any $\psi \in \mathrm{Irr}(\mathbb H_m)$ and $\psi' \in \mathrm{Irr}(\mathbb H_{m-i})$,
\[  \mathrm{dim}~ \mathrm{Hom}_{\mathbb H_{m-i}}(\mathbf{BZ}_{\mathrm{triv}}(\psi), \psi') \leq 1 ,
\]
\[  \mathrm{dim}~ \mathrm{Hom}_{\mathbb H_{m-i}}(\mathbf{BZ}_{\mathrm{sgn}}(\psi), \psi') \leq 1 .
\]
\end{theorem}

\begin{proof}
The original version is stated for the sign representation, see \cite[Proposition 2.5]{Ch21} and \cite{Ch22+}. But one can then transfer to the trivial representation by using the Iwahori-Matsumoto dual. More precisely, the Iwahori-Matsumoto involution $IM_k: \mathbb H_k \rightarrow \mathbb H_k$ is given as:
\[    w \mapsto (-1)^{l(w)} w \quad \mbox{ for $w \in S_k$ }, \quad   y_j \mapsto -y_j   \quad \mbox{ for all $j$ } . \]
Then one checks that $IM_{n-i} \circ \mathbf{BZ}_{\mathrm{sgn}} \circ IM_n = \mathbf{BZ}_{\mathrm{triv}}$, where $\mathrm{sgn}$ and $\mathrm{triv}$ are the sign and trivial representations of $S_i$ respectively. Since those $IM$ defines a categorical equivalence, the result then follows from the sign case in \cite{CS19, Ch22+}. The multiplicity freeness comes from the Bernstein-Zelevinsky theory and the multiplicity one of branching law in \cite{AGRS10, Ch23}.
\end{proof}

 Using the functor $\Gamma_{n,m}$, we deduce that:

\begin{corollary} \label{cor tensor hom bound}
 Let $X$ be an irreducible Harish-Chandra module of $\mathrm{GL}_n(\mathbb C)$. Let $S^iV$ and $\wedge^iV$ be the $i$-th symmetric tensor representation and the $i$-th exterior tensor representation of $V$ respectively. Then, for any irreducible Harish-Chandra module $X'$ of $\mathrm{GL}_n(\mathbb C)$,
\[  \mathrm{dim}~ \mathrm{Hom}_{\mathrm{GL}_n(\mathbb C)}(X \otimes \wedge^iV, X') \leq 1
\]
\[  \mathrm{dim}~\mathrm{Hom}_{\mathrm{GL}_n(\mathbb C)}(X \otimes S^iV, X') \leq 1 .
\]
\end{corollary}

\begin{proof}
We only explain the statement for $\wedge^iV$ and the one for $S^iV$ is very similar. We shall consider $k$ to be sufficiently large such that $\chi^{\otimes k}\otimes \pi'$ is thickened. Let $\widetilde{X}'=\chi^{\otimes k}\otimes X'$ and let $\widetilde{X}=\chi^{\otimes k} \otimes X$. Then, it suffices to show that
\[ \mathrm{dim}~ \mathrm{Hom}_{\mathrm{GL}_n(\mathbb C)}(\widetilde{X} \otimes \wedge^iV, \widetilde{X}') \leq 1.
\]
By the exactness of $\Gamma_{n,m}$, we have:
\[   \mathrm{dim}~ \mathrm{Hom}_{\mathrm{GL}_n(\mathbb C)}(\widetilde{X} \otimes \wedge^iV, \widetilde{X}')\leq \mathrm{dim}~\mathrm{Hom}_{\mathbb H_{m-i}}(\Gamma_{n,m-i}(\widetilde{X}\otimes \wedge^iV), \Gamma_{n,m-i}(\widetilde{X}')) .
\]
Recall that $\mathbb{S}_{\mathrm{triv}}(V)=\wedge^iV$ by the Schur-Weyl duality. But now, by Theorems \ref{thm bz undr arakawa suzuki} and \ref{thm bz multiplicity one}, the latter one has dimension at most one, as desired.
\end{proof}

We remark that \cite{Ch22+d} also determines (up to applying the Iwahori-Matsumoto involution) when the inequalities in Theorem \ref{thm bz multiplicity one}  are equalities. This in particular gives a list of possible simple quotients for $X \otimes \wedge^iV$ and $X \otimes S^iV$. One may expect that the list should exhaust those simple quotients.

Incorporating the improved multiplicity result in \cite[Theorem 1]{Ch23}, we also have:

\begin{corollary} \label{cor tensor hom bound std}
Let $X$ be a standard module in $\mathcal{HC}_n$. Then for any irreducible Harish-Chandra module $X'$ of $\mathrm{GL}_n(\mathbb C)$, 
\[ \mathrm{dim}~ \mathrm{Hom}_{\mathrm{GL}_n(\mathbb C)}(X \otimes S^iV, X') \leq 1.
\]
\end{corollary}

\begin{remark}
While Corollaries \ref{cor tensor hom bound} and \ref{cor tensor hom bound std} are stated for our choice of $V$, one can apply the Hermitian involution or the involution from $g\mapsto g^{-t}$ to get to another standard representations as well as their contragredients.
\end{remark}

\begin{remark} \label{rmk tensor product}
We explain an instance how the tensor product in Corollary \ref{cor tensor hom bound} comes into the play of branching laws and this suggests some kind of Lefschetz principle in that direction. Let $\mathcal S(\mathbb C^n)$ and $\mathcal S(\mathbb C^n-0)$ be the spaces of Schwartz functions on $\mathbb C^n$ and $\mathbb C^n-0$ respectively. The branching law for the equal rank Fourier-Jacobi model is to study simple $\mathrm{GL}_n(\mathbb C)$-quotients of 
\[  \pi \hat{\otimes} \mathcal S(\mathbb C^n) 
\]
for some irreducible representation $\pi$ of $\mathrm{GL}_n(\mathbb C)$. Here $\hat{\otimes}$ is the completed projective tensor product. Using Borel's lemma, we have a natural short exact sequence for $\mathrm{GL}_n(\mathbb C)$-representations,
\[   0 \rightarrow \mathcal S(\mathbb C^n-0) \rightarrow \mathcal S(\mathbb C^n) \rightarrow \mathbb C[[z_1, \ldots, z_n, \bar{z}_1, \ldots, \bar{z}_n]] \rightarrow 0
\]
and so $ 0 \rightarrow \mathcal S(\mathbb C^n-0) \hat{\otimes} \pi \rightarrow \mathcal S(\mathbb C^n)\hat{\otimes}\pi \rightarrow \mathbb C[[z_1, \ldots, z_n, \bar{z}_1, \ldots, \bar{z}_n]]\hat{\otimes} \pi \rightarrow 0$.  In particular, the latter tensor product admits a filtration of the form $\pi \otimes S^iV \otimes S^jV'$ (which one may interpret as an analogue of a layer of Bernstein-Zelevinsky filtration in \cite{Ch22+d}) and so we expect \cite{Ch22+, Ch22+d} will be useful in determining those simple quotients. Here $V$ and $V'$ are the conjugate standard and standard representations of $\mathrm{GL}_{n}(\mathbb C)$ respectively.
\end{remark}

\section{Applications on parabolic inductions} \label{s application on parabolic}
\subsection{Reducibility of parabolic inductions}






Recall that $\chi$ is defined in Section \ref{ss defining thickening character}. Determining the irreducibility of parabolic inductions for Harish-Chandra modules can be transferred to the same problem for the graded Hecke algebra side in the following Corollary \ref{cor detect parabolic induction} and the precise way to do so is indicated in its proof.

\begin{corollary} \label{cor detect parabolic induction}
The irreducibility of parabolic induction for the Harish-Chandra category $\mathcal{HC}_n$ can be detected from $\mathcal H_m$ (for some $m$) by using $\Gamma_{n,m}$.
\end{corollary}

\begin{proof}
Let $X$ and $Y$ be two irreducible Harish-Chandra modules of $\mathrm{GL}_{n_1}(\mathbb C)$ and $\mathrm{GL}_{n_2}(\mathbb C)$ respectively. We can choose sufficiently large $k$ such that for any composition factor $Z$ in $X \times Y$, $\chi^{\otimes k} \otimes Z$ is thickened. Since $\chi^{\otimes k} \otimes$ defines an equivalence of categories, it is clear that $\chi^{\otimes k}\otimes (X\times Y)$ is irreducible if and only if $X \times Y$ is irreducible. Now, let $m$ be the height of $Z$ for some (and so all) composition factor $Z$ in $\chi^{\otimes k}\otimes (X \times Y)$. Since $\Gamma_{n_1+n_2, m}$ sends any simple composition factor of $\chi^{\otimes k}\otimes (X \times Y)$ to a non-zero irreducible module (Theorem \ref{thm irreducibility of functor}), we then have that $\chi^{\otimes k}\otimes (X\times Y)$ is irreducible if and only if $\Gamma_{n_1+n_2,m}(\chi^{\otimes k} \otimes (X\times Y)) \cong \Gamma_{n_1,m_1}(\chi^{\otimes k}\otimes X)\times \Gamma_{n_2,m_2}(\chi^{\otimes k}\otimes Y)$ is irreducible. Here $m_1=\mathrm{ht}(\chi^{\otimes k} \otimes X)$ and $m_2=\mathrm{ht}(\chi^{\otimes k}\otimes Y)$, and the last isomorphism follows from Theorem \ref{thm isomophic of parabolic induction} and Lemma \ref{lem height zero}.
\end{proof}

For example, Corollary \ref{cor detect parabolic induction} gives another perspective on the classical result that the product of two irreducible unitary representations is still irreducible \cite[Theorem 3.1]{Sa89} (c.f. \cite[Theorem 6.18(b)]{Vo86} and \cite{Bu03}).

Recall in Section \ref{s lefeschetz principle dirac} that, for a finite-dimensional representation $X$ of $\mathrm{GL}_n(\mathbb C)$, $X \cong J(\lambda_L, \lambda_R)$ for some regular dominant $\lambda_L$ and $\lambda_R$ in $\mathfrak h_0^*$. In particular, $\Gamma_{n,m}(X)$ is a ladder representation if it is non-zero.

\begin{corollary} \label{cor parabolic induction preserve}
Let $X_1$ and $X_2$ be irreducible representations in $\mathcal{HC}_{n_1}$ and $\mathcal{HC}_{n_2}$ respectively. Suppose $X_1$ or $X_2$ is finite-dimensional. Then $X_1 \times X_2$ has unique simple quotient and unique simple submodule. 
\end{corollary}

\begin{proof}
Again, we find a sufficiently large $k$ such that $\chi^{\otimes k}\otimes Z$ is thickened for any simple composition factor $Z$ in $X_1 \times X_2$. We have that $\chi^{\otimes k}\otimes (X_1 \times X_2)\cong (\chi^{\otimes k}\otimes X_1)\times (\chi^{\otimes k}\otimes X_2)$, and so by Theorem \ref{thm isomophic of parabolic induction}, 
\[  \Gamma_{n,m}(\chi^{\otimes k}\otimes (X_1\times X_2)) \cong \Gamma_{n,m_1}(\chi^{\otimes k}\otimes X_1) \times \Gamma_{n,m_2}(\chi^{\otimes k}\otimes X_2) ,
\]
where $m_1$ and $m_2$ are the heights of $\chi^{\otimes k}\otimes X_1$ and $\chi^{\otimes k}\otimes X_2$ respectively. 

Now $\Gamma_{n,m_2}(\chi^{\otimes k}\otimes X_2)$ is a ladder representation, and so the product has a unique simple quotient \cite{LM16} (with the translation via  \cite{Bo76} and \cite{Lu89}). Since $\Gamma_{n,m}$ is exact and sends all the composition factors in $\chi^{\otimes k}\otimes (X_1\times X_2)$ to a non-zero module, we then must have $\chi^{\otimes k}\otimes (X_1 \times X_2)$ has a unique simple quotient and so does $X_1 \times X_2$. The statement for submodule is the same.
\end{proof}

We also have the following variation of Corollary \ref{cor parabolic induction preserve}:

\begin{corollary} \label{cor jacquet functor preserve}
Let $X$ be an irreducible Harish-Chandra representation in  $\mathcal{HC}_n$. Let $Y$ be an irreducible finite-dimensional representation in $\mathcal{HC}_{n'}$ with $n'<n$. Then there exists at most one irreducible Harish-Chandra module $Z$ of $\mathrm{GL}_{n-n'}(\mathbb C)$ such that $X$ is the unique simple submodule of $ Y \times Z$. The statement also holds if one replaces $Y \times Z$ by $Z \times Y$.
\end{corollary}

\begin{proof}
The analogous statement for $\mathbb H_m$-modules holds for the product between a ladder representation and an arbitrary irreducible module. The argument for passing from Harish-Chandra modules to $\mathbb H_m$-modules is similar to the one in Corollary \ref{cor parabolic induction preserve}. Here we have also used Remark \ref{rmk irred}(3) that after some thickening for irreducible $X_1$ and $X_2$, $\Gamma_{n,m}(X_1)\cong \Gamma_{n,m}(X_2)\neq 0$ implies $X_1 \cong X_2$.
\end{proof}

When $Y$ is a character, there exists algorithms (e.g. \cite{LM16}) to compute such $Y$ in Corollary \ref{cor jacquet functor preserve}.

\begin{remark} \label{rmk deduce branching from jacquet}
Here we explain how one can deduce some new explicit quotient branching laws from results of parabolic inductions. Let $X$ be an irreducible Casselman-Wallach representation of $\mathrm{GL}_{n+1}(\mathbb C)$. Given a character $\mu$ of $\mathrm{GL}_1(\mathbb C)$, one can then find, using results of $p$-adic group side (c.f. Corollary \ref{cor jacquet functor preserve}), the Langlands parameter of an irreducible Harish-Chandra module $Y$ of $\mathrm{GL}_n(\mathbb C)$ (if such $Y$ exists) such that 
\[   \mathrm{Hom}_{\mathrm{GL}_{n+1}(\mathbb C)}(X, Y \times \mu) \neq 0 .
\]
Via Frobenius reciprocity, one then has:
\[  \mathrm{Hom}_{P_{n,1}}(X, \delta^{1/2}(Y \boxtimes \mu)) \neq 0 ,
\]
Then restricting to $\mathrm{GL}_n(\mathbb C)$ (viewed as a subgroup of $P_{n,1}$ via the embedding $g \mapsto \begin{pmatrix} g& \\ & 1 \end{pmatrix}$), we obtain a branching law:
\[   \mathrm{Hom}_{\mathrm{GL}_n(\mathbb C)}(X, \delta^{1/2} Y) \neq 0 .
\]
For example, let $X=J((2,2,1), (1,1,0))$. Then $J((2,1),(1,0)) \times J((2),(1))$ has $X$ as a simple submodule and so $\delta^{1/2} Y=J((\frac{5}{2},\frac{3}{2}),(\frac{3}{2},\frac{1}{2}))$ is a simple quotient of $X|_{\mathrm{GL}_2(\mathbb C)}$.


\end{remark}

\begin{remark}
 There are other results of parabolic inductions (or in the form of Jacquet functors) in e.g. \cite{LM16, Ch22+b, Ch22+c, Ch22+d} and references therein. The reader is invited to translate those results and find applications.
\end{remark}

\subsection{Remarks on higher structure} \label{ss remark on higher structures}

As seen in \cite{Ch22+b}, it is useful to determine some higher structure for branching problems. However, for such applications, we need to show that $\Gamma_{n,m}$ preserves some higher structure to transfer the results in \cite{Ch22+b} to the Harish-Chandra category. Another application of preserving higher structure is determining the equalities in Corollary \ref{cor tensor hom bound}. This question for the Arakawa-Suzuki functor \cite{AS98} is studied in \cite{Fu18}, which shows a fully-faithful embedding from the deformed BGG category to the completion of module category of $\mathbb H_m$. We leave our case for future investigation.

\appendix
\section{Dirac cohomology of finite dimensional representations} \label{appendix}
In this section, we study the Dirac cohomology
of finite dimensional representations for all 
complex connected reductive Lie groups, which includes $\mathrm{GL}_n(\mathbb C)$ as a special case. We shall use the notations in \cite[Section 2]{BP}, which in particular coincide with notations for the $\mathrm{GL}_n(\mathbb C)$ case. Firstly, an irreducible Harish-Chandra module  $J(\lambda_L,\lambda_R)$ is finite-dimensional if and only if
$\lambda_L$, $\lambda_R \in \mathfrak{h}_0^*$ are regular and integral. Moreover, its $K$-spectrum is equal to:
\begin{equation} \label{eq-finited}
J(\lambda_L,\lambda_R)|_K = F_{\lambda_L - \rho} \otimes F_{\lambda_R - \rho}^* = F_{\lambda_L - \rho} \otimes F_{-w_0(\lambda_R - \rho)}
\end{equation}
In particular, the PRV component of the above tensor product has extremal weight $(\lambda_L - \rho) + w_0(-w_0(\lambda_R-\rho)) = \lambda_L - \lambda_R$, which is equal to that of the lowest $K$-type of $J(\lambda_L,\lambda_R)$.

The main result of \cite{HP} (reformulated in Equation (2.2) of \cite{BP} for complex groups) implies that all finite dimensional representations of complex groups with nonzero Dirac cohomology must be of the form $X = J(\lambda,-w_0\lambda)$, where $\lambda$ is regular integral and $w_0 \in W$ is the longest element in the Weyl group. 

Now consider the \emph{Dirac operator} 
$$D: X \otimes S \longrightarrow X \otimes S,$$
where $S$ is the spinor module, which is isomorphic to $2^{\lfloor n/2 \rfloor} F_{\rho}$ as a $\widetilde{K}$-module by \cite[Lemma 2.2]{BP}. By Weyl's unitarity trick, there is a inner product $\langle , \rangle_{X}$ on $\pi$ such that $\langle Ev, v' \rangle_{X} = -\langle v, Ev' \rangle_{X}$ for all $E \in \mathfrak{k}_0 + j\mathfrak{s}_0$ and $v, v' \in X$ (note that
this defines a positive definite Hermitian form with respect to the anti-involution given in Remark \ref{rmk-bullet}). 
Consequently, as in Remark 3.2.4 of \cite{HP2}, one can define an inner product on $X \otimes S$ such that $D$
is \emph{skew}-Hermitian with respect to this inner product.

As a result, the Dirac cohomology of $\pi$ is equal to $\ker(D) = \ker(D^2)$,
and the question of finding Dirac cohomology of $\pi$ reduces to studying
$D^2: X \otimes S \to X \otimes S$, where
$X \otimes S \cong 2^{\lfloor n/2 \rfloor} F_{\lambda - \rho} \otimes F_{\lambda - \rho} \otimes F_{\rho}$ as $\widetilde{K}$-modules.

By the formula of $D^2$ 
given in \cite[Proposition 3.1.6]{HP2} for instance, $D^2$
acts on the $\gamma$-isotypic component of $X \otimes S$ by the scalar $-||2\lambda||^2+||\gamma + \rho||^2$. Then it is easy to check that
the only $\gamma$-isotypic component which
$D^2$ acts by zero is when $\gamma = (\lambda - \rho)+(\lambda - \rho) + \rho$,
which implies the following:
\begin{theorem} \label{thm-finitedim}
Let $G$ be any complex connected reductive Lie group. Then all finite dimensional $(\mathfrak{g},K)$-modules with nonzero Dirac cohomology  must be of the form $X = J(\lambda,-w_0\lambda)$ for some regular integral $\lambda$. In particular, its Dirac cohomology is equal to $2^{\lfloor n/2 \rfloor} F_{2\lambda - \rho}$.
\end{theorem}

Finally, we remark that the {\it twisted Dirac index} of $X = J(\lambda,-w_0\lambda)$ can also be obtained easily. Namely, the character formula of $X$ is given by:
\begin{equation} 
\label{eq char form}
\sum_{w \in W} \det(w)X(\lambda,-ww_0\lambda). \end{equation}
By Theorem 5.1 of \cite{BPT}, only the last summand $w = w_0$ in the above formula contributes to the twisted Dirac index of $X$, and is equal to $\pm 2^{\lfloor n/2 \rfloor}F_{2\lambda - \rho}$ (the $\pm$ sign depends on the choice of $\epsilon$ in the theorem). This gives another evidence on the validity of Theorem \ref{thm-finitedim}.


\begin{thebibliography}{B1}
\bibitem[ALTV20]{ALTV20} J. Adams, M. van Leeuwen, P. Trapa, D. Vogan, Unitary representations of real reductive groups, Ast\'erisque, {\bf 417}, 2020.

\bibitem[AGRS10]{AGRS10}  A. Aizenbud, D. Gourevitch, S. Rallis and G. Schiffmann, Multiplicity one theorems, Ann. of Math. (2) {\bf 172} (2010), no. 2, 1407-1434.

\bibitem[AS98]{AS98} T.~Arakawa, T.~Suzuki, Duality between $\frs\frl_n({\bf C})$ and the degenerate affine Hecke algebra, J. Algebra {\bf 209} (1998), no. 1, 288-304.

\bibitem[Ba89]{Ba89} D.~Barbasch, The unitary dual for complex classical Lie groups, Invent. Math. {\bf 96} (1989), no. 1, 103-176.

\bibitem[Ba10]{B10} D.~Barbasch, The unitary spherical spectrum for split classical groups, J. Inst. Math. Jussieu {\bf 9} (2010), no. 2, 265-356.

\bibitem[BC14]{BC14} D.~Barbasch, D.~Ciubotaru, Unitary Hecke algebra modules with nonzero Dirac cohomology, Progr. Math., {\bf 257} (2014), Birkh\"auser/Springer, 1-20

\bibitem[BC15]{BC15} D.~Barbasch, D.~Ciubotaru,   Ladder representations of $\mathrm{GL}(n,\mathbb Q_p)$. In: Nevins, M., Trapa, P. (eds) Representations of Reductive Groups. Progress in Mathematics {\bf 312} (2015) Birkh\"auser, Cham. 

\bibitem[BCT12]{BCT} D.~Barbasch, D.~Ciubotaru, P.~Trapa, Dirac cohomology for graded affine Hecke algebras, Acta
Math. {\bf 202} (2012), no. 2, 197-227.

\bibitem[BP11]{BP}
D.~Barbasch, P. Pand\v zi\'c,
Dirac cohomology and unipotent representations of complex groups,
Noncommutative geometry and global analysis,
Contemp. Math. \textbf{546}, Amer. Math. Soc., Providence, RI, 2011, 1--22.

\bibitem[BPT19]{BPT}
D.~Barbasch, P. Pand\v zi\'c, P.~Trapa,
Dirac index and twisted characters, Trans. Amer. Math. Soc. {\bf 371}, Number 3 (2019) 1701-1733.


\bibitem[BKZ09]{BKZ09} L. Barchini, A. C. Kable, R. Zierau, Conformally invariant systems of differential operators, Adv. Math. {\bf 221} Issue 3 (2009), 788-811. doi.org/10.1016/j.aim.2009.01.006.

\bibitem[Bu03]{Bu03} Baruch, E.M., Proof of Kirillov's conjecture, Ann. of Math. {\bf 158} (2003), 207-252.

\bibitem[BG80]{BG80} J. Bernstein and S. Gelfand, Tensor products of finite and infinite dimensional representations of semisimple Lie algebras, Compositio Math. {\bf 41} (1980), 245-285. 

\bibitem[BK14]{BK14} Bernstein, J., Kr\"otz, B.: Smooth Fr\'echet globalizations of Harish–Chandra modules. Isr. J. Math. 199, 45-111 (2014). https://doi.org/10.1007/s11856-013-0056-1





\bibitem[Bo76]{Bo76} A. Borel, Admissible representations of a semi-simple group over a local field with vectors fixed under an iwahori subgroup. Invent Math {\bf 35} (1976), 233-259. https://doi.org/10.1007/BF01390139


\bibitem[Ca22]{Ca22} K. Calvert, Compact Schur-Weyl duality and the Type B/C VW-algebra, Adv. Math. {\bf 407} (2022), 108575. https://doi.org/10.1016/j.aim.2022.

\bibitem[Cas89]{Cas89}  W. Casselman, Canonical Extensions of Harish-Chandra Modules to Representations of G, Canadian Journal of Mathematics. 1989;41(3):385-438. doi:10.4153/CJM-1989-019-5 


\bibitem[Ch18]{Ch18} K.Y. Chan, 
A vanishing theorem for Dirac cohomology of standard modules. Adv. Math. {\bf 325} (2018), 274-311.

\bibitem[Ch21]{Ch21} K.Y. Chan, Homological branching law for $(\mathrm{GL}_{n+1}(F),\mathrm{GL}_n(F))$: projectivity and indecomposability, Invent. math. {\bf 225} (2021), 299-345. https://doi.org/10.1007/s00222-021-01033-5

\bibitem[Ch23]{Ch23} K.Y. Chan, Ext-multiplicity theorem for standard representations of $(\mathrm{GL}_{n+1}, \mathrm{GL}_n)$, Math. Z. {\bf 303}, 45 (2023). https://doi.org/10.1007/s00209-022-03198-y

\bibitem[Ch22+]{Ch22+} K.Y. Chan, Construction of simple quotients of Bernstein-Zelevinsky derivatives and highest derivative multisegments I: reduction to combinatorics, preprint (2022).

\bibitem[Ch22+b]{Ch22+b} K.Y. Chan, On the product functor on the inner forms of general linear group over a non-Archimedean local field, to appear in Transformation Groups.

\bibitem[Ch22+c]{Ch22+c} K.Y. Chan, Duality for generalized Gan-Gross-Prasad relevant pairs for $p$-adic $\mathrm{GL}_n$, arXiv:2210.17249 (2022)

\bibitem[Ch22+d]{Ch22+d} K.Y. Chan, Quotient branching law for $p$-adic $(\mathrm{GL}_{n+1}, \mathrm{GL}_n)$ I: generalized Gan-Gross-Prasad relevant pairs, arXiv:2212.05919 (2022)

\bibitem[CS19]{CS19} K.Y. Chan and G. Savin,  Bernstein-Zelevinsky derivatives: a Hecke algebra approach, International Mathematics Research Notices, {\bf 3} (2019), 731-760. https://doi.org/10.1093/imrn/rnx138



\bibitem[CT11]{CT11} D. Ciubotaru and P. Trapa, Functors for unitary representations of real classical groups and affine Hecke algebras, Adv. Math. {\bf 227} no. 4 (2011), 1585-1611.

\bibitem[CT12]{CT12} D. Ciubotaru and P. Trapa, Duality for $GL(n, \mathbb{R})$, $GL(n, \mathbb{Q}_p)$, and the degenerate affine Hecke algebra for $gl(n)$, Amer. J. Math. {\bf 134} (2012), 1-30.




\bibitem[DW20]{DW}
C.-P.~Dong and K.D.~Wong,
\emph{Scattered representations of $SL(n,\mathbb C)$}, Pacific J. Math. {\bf 309} (2020), no. 2, 289-312.

\bibitem[Du75]{D} M.~Duflo, Representations irreductibles des groupes semi-simples complexes, Lecture Notes in Math. {\bf 497} (1975), 26-88.



\bibitem[EFM09]{EFM08} P. Etingof, R. Freund, X. Ma, A Lie-theoretic construction of representations of the degenerate affine and double affine Hecke algebras of type BCn, Representation Theory {\bf 13} (2009), 33-49.

\bibitem[Ev96]{Ev96} S. Evens, The Langlands classification for graded Hecke algebras, Proc. Amer. Math. Soc., {\bf 124} (4) (1996),  1285-1290.

\bibitem[Fu18]{Fu18} R. Fujita, Tilting modules of affine quasi-hereditary algebras, Advances in Mathematics, {\bf  324} (2018) 241-266. doi.org/10.1016/j.aim.2017.11.013.



\bibitem[Go70]{Go70} R. Godement, Notes on Jacquet-Langlands' theory, preprint. Princeton (1970)

\bibitem[HP02]{HP}
J.-S. Huang and P. Pand\v{z}i\'{c}, Dirac cohomology, unitary representations and a proof of a conjecture of Vogan, J. Amer. Math. Soc. {\bf 15}(1) (2002), 185-202.

\bibitem[HP06]{HP2} J.-S. Huang, P. Pand\v{z}i\'{c}, Dirac Operators in Representation Theory, Mathematics: Theory and Applications, Birkhauser, 2006.


\bibitem[Hu93]{Hu93} J.-S. Huang,  Intertwining differential operators and reducibility of generalized Verma modules. Math. Ann. {\bf 297} (1993), 309-324. https://doi.org/10.1007/BF01459504





\bibitem[Knp86]{Knp86} A. Knapp, Representation Theory of Real Semisimple Groups: an Overview Based on Examples, Princeton University Press, Princeton, New Jersey, 1986.


\bibitem[LM16]{LM16} E. Lapid, A. M\'inguez, On parabolic induction on inner forms of the general linear group over a non-Archimedean local field,  Sel. Math. New Ser. {\bf 22} (2016), 2347-2400. 

\bibitem[Lu89]{Lu89} G. Lusztig, Affine Hecke algebras and their graded version, J. Amer. Math. Soc. {\bf 2} (1989), 599-685.

\bibitem[Od07]{Od07} H. Oda, Generalization of Harish-Chandra's basic theorem for Riemannian symmetric spaces of non-compact type, Adv. Math. {\bf 208 (2)} (2007) 549-596.


\bibitem[PRV67]{PRV67}  K.R. Parthasarathy, R. Ranga Rao, and V. S. Varadarajan. Representations of Complex Semi-Simple Lie Groups and Lie Algebras, Annals of Mathematics {\bf 85}, no. 3 (1967): 383-429. doi.org/10.2307/1970351.

\bibitem[Ro85]{Ro85} J.D. Rogawski, On modules over the Hecke algebra of a p-adic group, Invent. Math. {\bf 79}, 443-465 (1985). https://doi.org/10.1007/BF01388516

\bibitem[Ro86]{Ro86} J.D. Rogawski, Representations of $GL(n)$ over a p-adic field with an Iwahori-fixed vector, Israel J. Math. {\bf 54} (1986), 242-256.  https://doi.org/10.1007/BF02764944



\bibitem[Sa89]{Sa89} S. Sahi, On Kirillov's conjecture for archimedean fields, Compos. Math. {\bf 72} (1989), 67-86.

\bibitem[Su98]{Su98} T.~Suzuki, Rogawski's conjecture on the Jantzen filtration for the degenerate affine Hecke algebra of type A, Represent. Theory {\bf 2} (1998), 393-409

\bibitem[St67]{St67} E. M. Stein, Analysis in Matrix Spaces and Some New Representations of $SL(N,C)$, Annals of Mathematics, Second Series, {\bf 86}, No. 3 (1967),  461-490.

\bibitem[Ta86]{Ta86} M. Tadi\'c, On the classification of irreducible unitary representations of GL(n) and the conjectures of Bernstein and Zelevinsky, Ann. Sci. \'Ecole Norm. Sup., {\bf 19} (1986), 335-382.

\bibitem[Ta09]{Ta09} M. Tadi\'c, $\mathrm{GL}(n, \mathbb C)^{\wedge}$ and $\mathrm{GL}(n, \mathbb R)^{\wedge}$, Automorphic forms and L-functions II. Local aspects (2009),  285-313.


\bibitem[Tr06]{Tr06} F. Treves, Topological vector spaces, distributions and kernels, Mineola, N.Y.: Dover Publications. 



\bibitem[Vo81]{Vo81} D. Vogan, Representations of real reductive groups, Progr. Math., vol. 15, Birkh\"auser Boston, Boston, MA, 1981.
\bibitem[Vo86]{Vo86} D. Vogan,  The unitary dual of GL(n) over an archimedean field. Invent Math {\bf 83}, 449-505 (1986). https://doi.org/10.1007/BF01394418

\bibitem[Ze80]{Ze80} A. Zelevinsky, Induced representations of reductive p-adic groups II, Ann. Sci. Ecole Norm. Sup. {\bf 13} (1980), 154-210.



\bibitem[Ze85]{Ze85} A. Zelevinsky, Two remarks on graded nilpotent classes, Uspehi Mat. Nauk, 40 (1985), no. 1 (241), 199-200.



\bibitem[Zh74]{Zh} D.~P.~Zhelobenko, Harmonic analysis on complex semisimple Lie groups, Mir, Moscow, 1974.
\end{thebibliography}
\end{document}